\documentclass[12pt]{amsart}
\usepackage[a4paper,margin=2.5cm]{geometry}
\usepackage[T1]{fontenc}
\usepackage[english]{babel}
\usepackage{microtype}
\usepackage{comment}
\usepackage{amsmath,amssymb,amsthm,amscd,mathtools}
\usepackage{mathrsfs}
\usepackage{enumitem}
\usepackage{xcolor}
\usepackage{xcolor}
\definecolor{citeblue}{RGB}{0,70,140}
\usepackage[
  colorlinks=true,
  linkcolor=citeblue,
  citecolor=citeblue,
  urlcolor=citeblue
]{hyperref}
\usepackage{etoolbox}
\makeatletter
\patchcmd{\@setaddresses}
  {\par\addvspace\bigskipamount\indent}
  {\par\addvspace\bigskipamount\noindent}
  {}{}
\makeatother
\makeatletter
\def\l@subsection{\@tocline{2}{0pt}{3pc}{6pc}{}}
\makeatother
\makeatletter
\renewcommand\section{\@startsection{section}{1}%
  \z@{.7\linespacing\@plus\linespacing}{.5\linespacing}%
  {\normalfont\bfseries\centering}}

\renewcommand\subsection{\@startsection{subsection}{2}%
  \z@{.5\linespacing\@plus.7\linespacing}{-.5em}%
  {\normalfont\bfseries}}
\makeatother

\numberwithin{equation}{section}

\newtheorem{theorem}{Theorem}[section]
\newtheorem{prop}[theorem]{Proposition}
\newtheorem{lemma}[theorem]{Lemma}

\newtheorem*{remark}{Remark}

\DeclareMathOperator{\Tors}{Tors}
\DeclareMathOperator{\Cob}{Cob}

\title[ $U(1)$ Chern--Simons and Reshetikhin--Turaev TQFTs]{ Equivalence of Extended\\ $U(1)$ Chern--Simons and Reshetikhin--Turaev TQFTs}
\author{Daniel Galviz}
\address{YAU MATHEMATICAL SCIENCES CENTER AND DEPARTMENT OF MATHEMATICS, TSINGHUA UNIVERSITY, BEIJING, CHINA.}
\email{}
\date{}

\begin{document}

\begin{abstract}
\noindent
We establish the equivalence between \(U(1)\) Chern--Simons and Reshetikhin--Turaev TQFTs associated with finite quadratic modules. For gauge group \(U(1)\) and even level \(k\), we prove that the corresponding Chern--Simons TQFT is naturally isomorphic to the Reshetikhin--Turaev TQFT determined by the pointed modular category \(\mathcal C(\mathbb Z_k,q_k)\). The equivalence holds both for closed \(3\)-manifolds and for bordisms with boundary, so that the two constructions define naturally isomorphic extended \((2+1)\)-dimensional TQFTs. In particular, the finite quadratic module \((\mathbb Z_k,q_k)\) completely determines the \(U(1)\)  Chern--Simons theory.

\vspace{0.5cm}
\medskip

\noindent\textbf{Main Theorem.}
Let \(k\in\mathbb Z\) be a nonzero even integer, let \((\mathbb Z_k,q_k)\) be the associated finite quadratic module, and let
\(\mathcal C(\mathbb Z_k,q_k)\) denote the corresponding pointed modular category.
Then the Reshetikhin--Turaev theory associated with \(\mathcal C(\mathbb Z_k,q_k)\) is naturally isomorphic, as an extended \((2+1)\)-dimensional TQFT, to the Abelian Chern--Simons theory with gauge group \(U(1)\) and level \(k\).
Equivalently, the two theories define naturally isomorphic symmetric monoidal functors $\Cob^{\mathrm{ext}}_{2+1}\longrightarrow \mathrm{Vect}_{\mathbb C}.$
\end{abstract}

\maketitle
{\scriptsize
\setlength{\parskip}{0pt}
\hypersetup{linkcolor=black}
\tableofcontents}

\newpage
\section{Introduction}

Since the construction of Chern--Simons theory by Witten and its mathematical realization provided by Reshetikhin and Turaev \cite{Witten:1988,Reshetikhin1990}, an equivalence between these theories has long been expected, although it has not yet been proved. Historically, the non--Abelian theory was the first to attract attention; however, the Abelian case is by no means trivial. On the contrary, it provides a setting in which several constructions can be made completely explicit and compared in detail. Non--Abelian Chern--Simons theory with finite gauge group, developed  by Freed and Quinn \cite{FreedCS,Freed:1991}, already shows that topological gauge theories can often be understood by concrete cutting-and-gluing constructions. Abelian Chern--Simons theory may be viewed as a continuous counterpart of this picture, but one in which zero modes, torsion, and quadratic refinements play a more delicate role. Early gauge constructions of Abelian 3-manifold invariants in the $U(1)$ setting were already given by Gocho \cite{Gocho1,Gocho2}.

Two approaches to $U(1)$ Abelian Chern--Simons theory are especially relevant for this paper. The first is geometric. In a series of papers, Manoliu \cite{Manoliu1,Manoliu2,Manoliu3} studied the quantization of symplectic tori in real polarizations and then applied this machinery to $U(1)$ Chern--Simons theory on closed oriented $3$-manifolds and on manifolds with boundary. In her formulation, the partition function is expressed in terms of Ray--Singer/Reidemeister torsion, which makes both its topological invariance and its dependence on the torsion subgroup of $H_1(M;\mathbb Z)$ transparent. Just as importantly, Manoliu’s theory provides a genuine boundary-value formalism: the state spaces are obtained by geometric quantization of symplectic boundary tori, and bordisms act by operators constructed from Chern--Simons sections and Blattner--Kostant--Sternberg pairings \cite{Manoliu1,Manoliu2}.

The second approach is combinatorial. Mattes, Polyak, and Reshetikhin (MPR) \cite{Mattes}, and independently Murakami, Ohtsuki, and Okada \cite{Murakami}, constructed Abelian $3$-manifold invariants in terms of explicit quadratic Gauss sums associated with surgery presentations. In this framework, the dependence on linking numbers and quadratic forms on finite Abelian groups is completely explicit. This surgery description is closer in spirit to the Reshetikhin--Turaev formalism. Later work of Deloup and Turaev \cite{Deloup1999,deloup2005} further clarified the role of linking forms, reciprocity formulas, and finite quadratic data in these invariants.

The equivalence between the Abelian Reshetikhin--Turaev TQFT and the \(U(1)\) Chern--Simons theory has long been expected, but to our knowledge it had not previously been established. This equivalence is not immediate, because the two theories are formulated in very different languages. On the geometric side one encounters symplectic boundary phase spaces, real polarizations, and half-densities. On the combinatorial side one works with surgery presentations, Gauss sums, determinant corrections, and Kirby calculus. Many comparisons in the literature stop at the level of closed partition functions and do not address boundary state spaces or bordism operators \cite{Murakami,Jeffrey:1992, Mattes, Stirling, Guadagnini:2013sb,Guadagnini2014,Tagaris2025}. From the TQFT point of view, however, these are precisely the essential structures. Moreover, normalization issues are subtle. Geometric formulas contain Gaussian factors arising from the free part of cohomology together with torsion factors coming from Reidemeister or Ray--Singer torsion, whereas combinatorial formulas involve determinant factors and signature phases coming from quadratic reciprocity and surgery normalization. Any satisfactory comparison must therefore match these contributions term by term.

A further difficulty is that the degenerate surgery case $\det \mathcal L=0$, equivalently $b_1(M)>0$, cannot be ignored. From the perspective of TQFT functoriality, cutting and gluing naturally produce intermediate bordisms whose linking matrices are singular. Thus the correct statement is not simply an equality of partition functions for rational homology spheres, but an equivalence valid in the presence of both torsion sectors and zero-mode sectors. This is one reason why the natural framework for the problem is that of extended TQFT.

In Atiyah’s axiomatic language, a $3$-dimensional TQFT is a symmetric monoidal functor from a bordism category to vector spaces \cite{Atiyah:1989vu}. In the present setting, however, one must work with an extended bordism category in which boundary surfaces carry additional Lagrangian data and gluing is corrected by Maslov indices \cite{Walker,Turaev1994}. The correct formulation of the problem is therefore not merely to compare scalar invariants of closed $3$-manifolds, but to compare the extended $(2+1)$-dimensional theories. In concrete terms, one seeks canonically isomorphic state spaces for each extended surface $(\Sigma,\lambda)$, and identical linear maps for each extended bordism $X:(\Sigma_{\mathrm{in}},\lambda_{\mathrm{in}})\longrightarrow(\Sigma_{\mathrm{out}},\lambda_{\mathrm{out}}), $ after the Maslov anomaly has been normalized in a compatible manner.

The $U(1)$ case treated here provides the fundamental model for the general Abelian theory. Precisely because it is fully explicit, it allows a complete comparison between the geometric and combinatorial constructions at the TQFT level, including boundary state spaces and bordism operators. Related works develop the higher-rank toral case and its classification \cite{Galviz2,Galviz2.5,Galviz3,Galviz4}.

The paper proceeds as follows. First, we review Manoliu's geometric quantization construction of \(U(1)\) Chern--Simons theory, including the boundary state spaces, bordism operators, and extended gluing formalism. Second, we reformulate the corresponding Abelian Reshetikhin--Turaev theory for the pointed modular category \(C(\mathbb{Z}_k,q_k)\), identifying its surgery expressions and boundary state spaces. Third, we prove that these two constructions agree on closed \(3\)-manifolds, on boundary state spaces, and on the operators assigned to bordisms, yielding a natural isomorphism of extended \((2+1)\)-dimensional TQFTs. In particular, the finite quadratic module \((\mathbb{Z}_k,q_k)\) is shown to encode exactly the discrete data governing the \(U(1)\) Chern--Simons theory.

\medskip

\textbf{Acknowledgements.}
I would like to thank Nicolai Reshetikhin for many helpful conversations, suggestions, for introducing me to Abelian Chern–Simons theory, and for his support throughout the development of this project.  I would also like to thank Dan Freed, Stavros Garoufalidis, Zhengwei Liu, and Yilong Wang for their valuable comments and suggestions on the manuscript.

\section{\texorpdfstring{Geometric Quantization of $U(1)$  Chern--Simons Theory}{ Geometric Quantization of U(1)  Chern--Simons Theory}}
\label{sec:1}

In this section we summarize the main results of Manoliu concerning the construction of
Abelian \(U(1)\) Chern--Simons theory as a unitary extended
\((2+1)\)-dimensional topological quantum field theory. We restrict ourselves
to the ingredients needed later: the Chern--Simons functional, the geometric
quantization of the boundary phase space, the construction of the vectors
assigned to \(3\)-manifolds, and the extended gluing law. For full proofs of the $U(1)$  Chern-Simons theory, we refer to the original sources \cite{Manoliu1,Manoliu2}.

\subsection{The Abelian Chern--Simons Functional}
Let \(X\) be a closed oriented \(3\)-manifold, let \(P\to X\) be a principal
\(U(1)\)-bundle, and let \(\Theta\) be a connection on \(P\). Following
\cite{Manoliu2}, the Chern--Simons functional is defined by passing to the
associated \(SU(2)\)-bundle
\[
\hat P=P\times_{U(1)}SU(2)\longrightarrow X,
\]
where \(U(1)\hookrightarrow SU(2)\) is the diagonal embedding
\[
e^{i\theta}\longmapsto
\begin{pmatrix}
e^{i\theta} & 0\\
0 & e^{-i\theta}
\end{pmatrix}.
\]
The connection \(\Theta\) induces an \(SU(2)\)-connection \(\hat\Theta\) on
\(\hat P\), and one sets
\begin{equation}
S_{X,P}(\Theta)
=
\int_X \hat s^{\,*}\alpha(\hat\Theta)
\qquad (\mathrm{mod}\,1),
\end{equation}
where \(\hat s:X\to \hat P\) is any section and
\begin{equation}
\alpha(\hat\Theta)
=
\left\langle \hat\Theta\wedge F_{\hat\Theta}\right\rangle^{\flat}
-\frac16
\left\langle \hat\Theta\wedge[\hat\Theta,\hat\Theta]\right\rangle^{\flat}.
\end{equation}
Here \(\alpha(\hat\Theta)\) is the ordinary Chern--Simons \(3\)-form of the
\(SU(2)\)-connection \(\hat\Theta\). The value of \(S_{X,P}(\Theta)\in
\mathbb R/\mathbb Z\) is independent of the choice of \(\hat s\), since
different choices differ by an integer winding-number term.

The functional \(S_{X,P}:\mathcal A_P\to \mathbb R/\mathbb Z\) is gauge invariant and therefore descends to the quotient
\(\mathcal A_P/\mathcal G_P\). It satisfies the expected formal properties:
functoriality under orientation-preserving bundle morphisms,
\[
S_{X',P'}(\phi^*\Theta)=S_{X,P}(\Theta),
\]
orientation reversal,
\[
S_{-X,P}(\Theta)=-S_{X,P}(\Theta),
\]
and additivity under disjoint union,
\[
S_{X_1\sqcup X_2,\;P_1\sqcup P_2}(\Theta_1\sqcup \Theta_2)
=
S_{X_1,P_1}(\Theta_1)+S_{X_2,P_2}(\Theta_2).
\]
Moreover, its critical points are precisely the flat connections:
\begin{equation}
dS_{X,P}(\Theta)=0
\quad\Longleftrightarrow\quad
F_\Theta=0.
\end{equation}
Since the gauge group is Abelian, this condition is simply \(d\Theta=0\) in a
local trivialization.

When \(X\) has boundary, the Chern--Simons functional is no longer a
gauge-invariant scalar, but instead gives rise to a section of the boundary
prequantum line bundle. This is the basic geometric input in the construction of
the Abelian Chern--Simons state \(Z_X^{\mathrm{CS}}\).

\subsection{Geometric Quantization of the Boundary Phase Space}

Throughout, the gauge group is \(U(1)=\{z\in\mathbb C:\ |z|=1\}\), and the level \(k\in\mathbb Z\) is assumed to be even. This ensures that the Chern--Simons functional is well defined modulo \(\mathbb Z\) without additional spin structure and that the corresponding prequantum line bundle exists. Since \(U(1)\) is Abelian, every flat connection on a closed oriented surface \(\Sigma\) is determined by its holonomy, so the moduli space of flat \(U(1)\)-connections modulo gauge equivalence identifies canonically with the torus
\[\mathcal M_\Sigma\cong H^1(\Sigma;\mathbb R)/H^1(\Sigma;\mathbb Z). 
\]
If \(g\) is the genus of \(\Sigma\), then \(\mathcal M_\Sigma\) is a \(2g\)-dimensional symplectic torus, with symplectic form \(\omega_\Sigma\) induced by the cup-product pairing on \(H^1(\Sigma;\mathbb R)\), equivalently by integration over \(\Sigma\). For even \(k\), one constructs a  Hermitian line bundle \(\mathcal L_\Sigma\to\mathcal M_\Sigma\) with unitary connection of curvature \(\operatorname{curv}(\nabla_{\mathcal L_\Sigma})=-2\pi i\,k\,\omega_\Sigma\); this is the Chern--Simons prequantum line bundle. Although prequantum line bundles on \((\mathcal M_\Sigma,k\omega_\Sigma)\) form a torsor under \(H^1(\mathcal M_\Sigma;U(1))\), the Chern--Simons line bundle construction canonically selects \(\mathcal L_\Sigma\) \cite{Freed:1991}.

To quantize \(\mathcal M_\Sigma\), one chooses a rational Lagrangian subspace \(L\subset H^1(\Sigma;\mathbb R)\), that is, a Lagrangian subspace defined over \(\mathbb Q\). Such an \(L\) determines an invariant real polarization \(\mathcal P_L\) of \(\mathcal M_\Sigma\). Quantization is performed with half-densities: if \(\mathcal{BS}_{\mathcal P_L}\) denotes the Bohr--Sommerfeld set of the polarization and \(|\det(\mathcal P_L^*)|^{1/2}\) the corresponding half-density bundle, then the Hilbert space assigned to the extended surface \((\Sigma,L)\) is
\begin{equation}\label{eq:H-Sigma-L}
\mathcal H(\Sigma,L)
=
\bigoplus_{\Lambda\subset \mathcal{BS}_{\mathcal P_L}}
\Gamma_{\mathrm{flat}}\!\left(
\Lambda;\mathcal L_\Sigma\otimes |\det(\mathcal P_L^*)|^{1/2}
\right).
\end{equation}
Equivalently, one may view each summand as the one-dimensional space of covariantly constant sections over a Bohr--Sommerfeld leaf. The resulting Hilbert space is finite-dimensional and satisfies \(\dim\mathcal H(\Sigma,L)=|k|^g\). If \(L_1\) and \(L_2\) are two rational Lagrangians, geometric quantization produces a canonical unitary Blattner--Kostant--Sternberg operator \(F_{L_2L_1}:\mathcal H(\Sigma,L_1)\to\mathcal H(\Sigma,L_2)\). These operators do not compose strictly: their composition is twisted by the Maslov--Kashiwara index of triples of Lagrangian subspaces, and this defect is precisely what the extended formalism corrects.

Now let \(X\) be a compact oriented \(3\)-manifold with boundary. The restriction of flat connections induces a map \(r_X:\mathcal M_X\to\mathcal M_{\partial X}\), and on cohomology the corresponding map \(\dot r_X:H^1(X;\mathbb R)\to H^1(\partial X;\mathbb R)\) has image \(L_X:=\operatorname{Im}(\dot r_X)\), which is a rational Lagrangian subspace of \(H^1(\partial X;\mathbb R)\). Thus \(X\) canonically determines a boundary polarization \(\mathcal P_{L_X}\) and hence a Hilbert space \(\mathcal H(\partial X,L_X)\). The moduli space of flat \(U(1)\)-connections on \(X\) decomposes according to the torsion topological type of the underlying bundle as
\[\mathcal M_X=\bigsqcup_{p\in \Tors H^2(X;\mathbb Z)}\mathcal M_{X,p}.\]
For each torsion class \(p\), the component \(\mathcal M_{X,p}\) carries a canonically defined covariantly constant Chern--Simons section \(\sigma_{X,p}: \mathcal M_{X,p}\to r_X^*\mathcal L_{\partial X}|_{\mathcal M_{X,p}}\), whose phase is \(\exp(2\pi i k\,CS(p))\), where \(CS(p)\) is a quadratic refinement of the torsion linking pairing. Summing over all components gives 
\[\sigma_X=\sum_{p\in \Tors H^2(X;\mathbb Z)}\sigma_{X,p}.\]
The second ingredient in the construction of the state associated with \(X\) is the canonical half-density \(\mu_X\) on the Lagrangian subtorus \(\Lambda_X:=\operatorname{Im}(r_X)\subset \mathcal M_{\partial X}\), obtained from Reidemeister torsion, equivalently Ray--Singer torsion. More precisely, the torsion \(T_X\), together with Poincar\'e duality and the long exact sequence of the pair \((X,\partial X)\), determines an invariant half-density on \(\Lambda_X\) denoted by \[\mu_X=\int_{H^1(X,\partial X;U(1))}(T_X)^{1/2}.\] Combining the Chern--Simons section and the torsion half-density yields the state
\begin{equation}\label{eq:ZX-CS-merged}
Z_X^{\mathrm{CS}}
=
\frac{k^{m_X}}{\#\Tors H^2(X;\mathbb Z)}
\sum_{p\in \Tors H^2(X;\mathbb Z)}
\sigma_{X,p}\otimes \mu_X
\in \mathcal H(\partial X,L_X),
\end{equation}
where
\[m_X=\frac14\big(\dim H^1(X;\mathbb R)+\dim H^1(X,\partial X;\mathbb R)-\dim H^0(X;\mathbb R)-\dim H^0(X,\partial X;\mathbb R)\big).\]
In the closed case \(\partial X=\varnothing\), the Hilbert space reduces to \(\mathbb C\), the half-density becomes a density on \(\mathcal M_X=H^1(X;U(1))\), and \(Z_X^{\mathrm{CS}}\) becomes a scalar invariant, equivalently 

\begin{equation}\label{closed-cs}
Z_X^{\mathrm{CS}}=k^{m_X}\int_{\mathcal M_X}\sigma_X\,(T_X)^{1/2},
\end{equation}
with \(m_X=\frac12(\dim H^1(X;\mathbb R)-\dim H^0(X;\mathbb R))\).

To obtain strict functoriality under gluing, one passes to Walker’s category of extended manifolds \cite{Walker}. An extended \(2\)-manifold is a pair \((\Sigma,L)\), where \(\Sigma\) is a closed oriented surface and \(L\subset H^1(\Sigma;\mathbb R)\) is a rational Lagrangian subspace. An extended \(3\)-manifold is a triple \((X,L,n)\), where \(X\) is a compact oriented \(3\)-manifold, \(L\subset H^1(\partial X;\mathbb R)\) is a rational Lagrangian subspace, and \(n\in\mathbb Z/8\mathbb Z\) records the Maslov correction required for strict gluing. The extended Abelian Chern--Simons TQFT then assigns the Hilbert space \(\mathcal H(\Sigma,L)\) to \((\Sigma,L)\), and to an extended \(3\)-manifold \((X,L,n)\) it assigns the vector
\begin{equation}\label{eq:extended-CS-state}
Z^{\mathrm{CS}}_{(X,L,n)}
=
e^{\frac{\pi i}{4}n}\,F_{LL_X}\!\left(Z_X^{\mathrm{CS}}\right)
\in \mathcal H(\partial X,L),
\end{equation}
where \(F_{LL_X}:\mathcal H(\partial X,L_X)\to \mathcal H(\partial X,L)\) is the BKS operator associated with the change of polarization. With this correction, the resulting assignments satisfy functoriality, multiplicativity under disjoint union, compatibility with orientation reversal, and the gluing law; see \cite{Manoliu2} for the full proof.

\subsection{Abelian Chern--Simons TQFT}
The Abelian Chern--Simons TQFT construction can  be summarized as follows:
\begin{theorem}[\cite{Manoliu2}, Theorem VI.11]
The assignments
\[
(\Sigma,L)\longmapsto \mathcal H(\Sigma,L),
\qquad
(X,L,n)\longmapsto Z^{\mathrm{CS}}_{(X,L,n)}
\]
\emph{define a unitary extended \((2+1)\)-dimensional TQFT. More precisely:}

\begin{enumerate}
\item \emph{\textbf{Functoriality:}} extended diffeomorphisms induce unitary isomorphisms
between the corresponding Hilbert spaces, and the vectors \(Z^{\mathrm{CS}}_{(X,L,n)}\) are
natural with respect to these maps.

\item \emph{\textbf{Orientation:}} there is a natural antilinear identification
\[
\mathcal H(-\Sigma,L)\cong \overline{\mathcal H(\Sigma,L)},
\]
and
\[
Z^{\mathrm{CS}}_{(-X,L,n)}=\overline{Z^{\mathrm{CS}}_{(X,L,n)}}.
\]

\item \emph{\textbf{Disjoint union:}} one has canonical unitary identifications
\[
\mathcal H(\Sigma_1\sqcup \Sigma_2,L_1\oplus L_2)
\cong
\mathcal H(\Sigma_1,L_1)\otimes \mathcal H(\Sigma_2,L_2),
\]
and the corresponding states are multiplicative under disjoint union.

\item \emph{\textbf{Cylinder axiom:}} the cylinder \(\Sigma\times I\) defines the identity
operator on \(\mathcal H(\Sigma,L)\).

\item \emph{\textbf{Gluing:}} if \(X\) is cut along a closed oriented surface
\(\Sigma\), then the state of the glued manifold is obtained by contracting the
state of the cut manifold using the Hermitian pairing on
\(\mathcal H(\Sigma,L_\Sigma)\), with the required Maslov correction built into
the extended structure.
\end{enumerate}
\end{theorem}

\medskip

Thus \(U(1)\) Abelian Chern--Simons theory at even level \(k\) gives a unitary extended \((2+1)\)-dimensional TQFT. Its boundary state spaces arise by
geometric quantization of symplectic moduli tori, and its bordism operators are
constructed from the Chern--Simons functional, BKS pairings, and torsion
half-densities. This geometric framework will be the geometric side of the
comparison with Abelian Reshetikhin--Turaev theory carried out in the next
sections.

\section{Surgery Invariants from Abelian Gauge Groups}
\label{sec:abelian-surgery}

This section reviews the Abelian surgery invariants introduced by Mattes, Polyak, and Reshetikhin \cite{Mattes}, and also  Murakami, Ohtsuki, and Okada \cite{Murakami},  with emphasis on the case relevant for the comparison with \(U(1)\) Chern--Simons theory. We focus on three points. First, we describe the pointed Reshetikhin--Turaev invariant associated with the finite quadratic data \((\mathbb Z_k,q_k)\), where \(k\in 2\mathbb Z_{>0}\). Second, we recall the Gauss reciprocity formulas that relate this surgery invariant to the decomposition of the Abelian Chern--Simons partition function into topological sectors. Third, we explain why an analogous construction based directly on the compact group \(U(1)\) does not yield a nontrivial Abelian Reshetikhin--Turaev theory under natural continuity assumptions.

The relevance of this discussion is that the combinatorial invariant attached to \((\mathbb Z_k,q_k)\) is the correct discrete counterpart of \(U(1)\) Chern--Simons theory at level \(k\). In later sections we will compare this invariant, first on closed \(3\)-manifolds and then in the presence of boundary, with the geometric theory reviewed above.

\subsection{\texorpdfstring{The Abelian RT invariant for \(G=\mathbb Z_k\)}{The Abelian RT invariant for G=Zk}}

Let \(k\in 2\mathbb Z_{>0}\), and write \(G=\mathbb Z_k=\mathbb Z/k\mathbb Z\) additively. To match the \(U(1)\) Chern--Simons theory at level \(k\), we fix the pointed ribbon data
\begin{equation}\label{eq:Zk-data}
F\equiv 1,
\qquad
q_k(x)=\exp\!\Big(\frac{\pi i}{k}x^2\Big),
\qquad
\Omega_k(x,y)=\frac{q_k(x+y)}{q_k(x)q_k(y)}
=\exp\!\Big(\frac{2\pi i}{k}xy\Big).
\end{equation}
Since \(k\) is even, \(q_k\) is well defined on \(\mathbb Z_k\), and \(\Omega_k\) is a nondegenerate bicharacter. These data determine a pointed modular category $\mathcal C_k:=\mathcal C(\mathbb Z_k,q_k),$ whose simple objects are indexed by elements of \(\mathbb Z_k\), with tensor product induced by addition in \(\mathbb Z_k\), trivial associator, braiding
\[
c_{x,y}=\Omega_k(x,y)\,\mathrm{id}_{x+y},
\]
and ribbon twist
\[
\theta_x=q_k(x)\,\mathrm{id}_x.
\]

Let $J=J_1\cup\cdots\cup J_s\subset S^3$ be an oriented framed link, and let \(x=(x_1,\dots,x_s)\in \mathbb Z_k^s\) be a coloring of its components. If \(\mathcal L_J\) denotes the symmetric linking matrix of \(J\), with framings on the diagonal, then the corresponding pointed Reshetikhin--Turaev evaluation is
\begin{equation}\label{eq:RT-link-evaluation}
\mathrm{RT}_{\mathcal C_k}(J;x)
=
\exp\!\Big(\frac{\pi i}{k}\,\langle x,\mathcal L_Jx\rangle\Big),
\end{equation}
where for vectors \(u,v\) and a matrix \(A\) we write
\[
\langle u,v\rangle=u^{\!\top}v,
\qquad
\langle u,Av\rangle=u^{\!\top}Av.
\]
Indeed, each framing contributes the twist \(q_k(x_i)\), each crossing contributes the braiding \(\Omega_k(x_i,x_j)^{\pm1}\), and multiplying all local contributions yields the quadratic exponential \eqref{eq:RT-link-evaluation}. Now let $L=L_1\cup\cdots\cup L_m$ be a surgery link presenting the closed oriented \(3\)-manifold \(M_L\), and let $L'=L'_1\cup\cdots\cup L'_r$ be a colored link in the complement of \(L\). We write the linking matrix of \(L\sqcup L'\) in block form as
\[
\mathcal L_{L\sqcup L'}
=
\begin{pmatrix}
\mathcal L_L & \mathcal L_{LL'}\\[2mm]
\mathcal L_{LL'}^{\!\top} & \mathcal L_{L'}
\end{pmatrix}.
\]
If \(g\in\mathbb Z_k^m\) colors the surgery components and \(h\in\mathbb Z_k^r\) colors \(L'\), then
\begin{equation}\label{eq:RT-surgery-block}
\mathrm{RT}_{\mathcal C_k}(L\sqcup L';g,h)
=
\exp\!\Bigg(
\frac{\pi i}{k}\Big(
\langle g,\mathcal L_Lg\rangle
+
2\langle g,\mathcal L_{LL'}h\rangle
+
\langle h,\mathcal L_{L'}h\rangle
\Big)\Bigg).
\end{equation}

The Abelian surgery invariant is obtained by averaging over the labels on the surgery link and inserting the usual Kirby normalization. Using the normalized counting measure on \(\mathbb Z_k\), the invariant takes the form \begin{equation}\label{eq:Zk-surgery-invariant} Z_{\mathbb Z_k}(L'_h\subset M_L) = \mathcal N_k(L)\, \frac{1}{k^m} \sum_{g\in\mathbb Z_k^m} \mathrm{RT}_{\mathcal C_k}(L\sqcup L';g,h), \end{equation} where \(\mathcal N_k(L)\) is chosen so as to ensure invariance under Kirby moves. The required normalization is expressed in terms of the quadratic Gauss sums \begin{equation}\label{eq:apm} A_+(k):=\sum_{s\in\mathbb Z_k}\exp\!\Big(\frac{\pi i}{k}s^2\Big), \qquad A_-(k):=\sum_{s\in\mathbb Z_k}\exp\!\Big(-\frac{\pi i}{k}s^2\Big) = \overline{A_+(k)}. \end{equation} For even \(k\), the evaluation gives
\begin{equation}\label{eq:gauss-evaluation}
A_+(k)=k^{1/2} e^{\pi i/4},
\qquad
A_-(k)=k^{1/2}  e^{-\pi i/4}.
\end{equation}
Equivalently, if $a_\pm(k)$ denote the normalized Gauss sums used in the Mattes–Polyak– Reshetikhin presentation \cite{Mattes}, then  $A_\pm(k)=ka_\pm(k)$. If \(\sigma(\mathcal L_L)\) denotes the signature of the surgery matrix, then the Kirby normalization is \begin{equation}\label{eq:kirby-normalization}
\mathcal N^{\mathrm{raw}}_k(L)
=k^{-1/2}
A_+(k)^{\frac{-m-\sigma(\mathcal L_L)}{2}}
\,
A_-(k)^{\frac{-m+\sigma(\mathcal L_L)}{2}}.
\end{equation}
Substituting \eqref{eq:RT-surgery-block} into \eqref{eq:Zk-surgery-invariant} yields
\begin{equation}\label{eq:Zk-final}
\begin{aligned}
Z^{RT,\mathrm{raw}}_{\mathbb Z_k}(L'_h\subset M_L)
&=
A_+(k)^{\frac{-m-\sigma(\mathcal L_L)}{2}}
A_-(k)^{\frac{-m+\sigma(\mathcal L_L)}{2}}
 \\
& \times\sum_{g\in\mathbb Z_k^m}
\exp\!\Bigg(
\frac{\pi i}{k}\Big(
\langle g,\mathcal L_Lg\rangle
+
2\langle g,\mathcal L_{LL'}h\rangle
+
\langle h,\mathcal L_{L'}h\rangle
\Big)\Bigg).
\end{aligned}
\end{equation}

When \(L'=\varnothing\), this reduces to the closed invariant
\begin{equation}\label{eq:Zk-closed}
Z^{RT,\mathrm{raw}}_{\mathbb Z_k}( M_L)
= k^{-1/2}
A_+(k)^{\frac{-m-\sigma(\mathcal L_L)}{2}}
A_-(k)^{\frac{-m+\sigma(\mathcal L_L)}{2}}
\sum_{g\in\mathbb Z_k^m}
\exp\!\Big(\frac{\pi i}{k}\,\langle g,\mathcal L_Lg\rangle\Big).
\end{equation}

Let us briefly indicate why this is invariant under Kirby moves. Invariance under isotopy is immediate. Invariance under \(K_2\) follows from the fact that a handle slide replaces \(\mathcal L_L\) by a congruent matrix
\[
\mathcal L_L\longmapsto A^{\!\top}\mathcal L_LA,
\qquad A\in \mathrm{GL}(m,\mathbb Z),
\]
and the induced change of variables \(g\mapsto A^{-1}g\) is bijective on \((\mathbb Z_k)^m\). For \(K_1\), adjoining an unknot of framing \(\pm1\) produces an additional factor
\[\sum_{s\in\mathbb Z_k}\exp\!\Big(\pm\frac{\pi i}{k}s^2\Big)=A_\pm(k),
\]
which is canceled precisely by the normalization \eqref{eq:kirby-normalization}. Therefore \eqref{eq:Zk-final} defines a topological invariant of the surgery presentation. This invariant is the Reshetikhin--Turaev theory associated with the finite quadratic module \((\mathbb Z_k,q_k)\). It is the combinatorial theory that will later be compared with \(U(1)\) Chern--Simons theory.

\paragraph{\texorpdfstring{The case \(G=\mathbb R\)}{The case G=R}.}
Mattes, Polyak, and Reshetikhin also considered the Abelian surgery construction for the noncompact group \(G=\mathbb R\), see \cite{Mattes}. In this case one takes the standard continuous quadratic data
\[
q_{\mathbb R}(x)=e^{\pi i x^2},
\qquad
\Omega_{\mathbb R}(x,y)=\frac{q_{\mathbb R}(x+y)}{q_{\mathbb R}(x)q_{\mathbb R}(y)}
=e^{2\pi i xy},
\]
so that the link evaluation is again given by the same quadratic expression as in the finite case. If \(L=L_1\cup\cdots\cup L_m\) is a surgery link with linking matrix \(\mathcal L_L\), and \(L'\) is an auxiliary colored link with color \(h\in\mathbb R^r\), the corresponding invariant is formally obtained by replacing the finite average over \(\mathbb Z_k^m\) with an oscillatory Gaussian integral over \(\mathbb R^m\):
\[
Z_{\mathbb R}(L'_h\subset M_L)
=
b_+^{\frac{-m-\sigma(\mathcal L_L)}{2}}
b_-^{\frac{-m+\sigma(\mathcal L_L)}{2}}
\int_{\mathbb R^m}
\exp\!\Bigg(
\pi i\Big(
\langle x,\mathcal L_Lx\rangle
+
2\langle x,\mathcal L_{LL'}h\rangle
+
\langle h,\mathcal L_{L'}h\rangle
\Big)\Bigg)\,dx,
\]
where the normalization constants are the Fresnel integrals
\[
b_+=\int_{\mathbb R}e^{\pi i s^2}\,ds=e^{\pi i/4},
\qquad
b_-=\int_{\mathbb R}e^{-\pi i s^2}\,ds=e^{-\pi i/4},
\]
understood in the oscillatory sense. When \(\mathcal L_L\) is nondegenerate, this integral evaluates to the familiar Gaussian factor
\[
e^{\frac{\pi i}{4}\sigma(\mathcal L_L)}\,|\det \mathcal L_L|^{-1/2}
\]
times the phase obtained by completing the square, while in the degenerate case one obtains a distribution supported along the null directions. Thus the \(\mathbb R\)-theory should be viewed as the continuous Gaussian analogue of the finite \(\mathbb Z_k\)-theory. Although it is useful conceptually, it is not the combinatorial model relevant for the comparison with \(U(1)\) Chern--Simons theory at level \(k\), which is governed instead by the finite quadratic data \((\mathbb Z_k,q_k)\).

\subsection{Reciprocity formulas and the relation with \texorpdfstring{\(U(1)\)}{U(1)} Chern--Simons theory}
\label{subsec:reciprocity}

Mattes, Polyak, and Reshetikhin, and independently Murakami, Ohtsuki, and Okada observed that the RT Gauss sums appearing in \eqref{eq:Zk-closed} are closely related to the decomposition of the Abelian \(U(1)\) Chern--Simons partition function into topological sectors \cite{Murakami,Mattes}. In the Abelian case, unlike the simply connected non--Abelian case, one must sum over isomorphism classes of principal \(U(1)\)-bundles:
\[
Z_X^{\mathrm{CS}}
=
\sum_{ |P|\in H^2(X;\mathbb Z)}
\int_{\mathcal A_P/\mathcal G_P}
\exp\!\bigl(2\pi i k\,S_{X,P}(\Theta)\bigr)\,D\Theta,
\]
For a rational homology sphere, each bundle contains at most one gauge-equivalence class of flat connections, so this partition function reduces to a discrete sum over topological sectors. The surgery Gauss sums in \eqref{eq:Zk-closed} are therefore expected to match these flat-sector contributions term by term.

The bridge between the combinatorial and geometric pictures is provided by quadratic reciprocity. Let \(\mathcal L\) be a symmetric, nondegenerate integer \(m\times m\) matrix, and let \(r\ge 1\). Then one has the reciprocity formula \cite[Theorem 1]{deloup2005}:
\begin{equation}\label{eq:reciprocity-nondeg} \sum_{n\in\mathbb Z_r^m} \exp\!\Big(\frac{\pi i}{r}\,\langle n,\mathcal Ln\rangle\Big)
=
r^{m/2}
\Big(\det(\mathcal L/i)\Big)^{-1/2}
\sum_{l\in \mathbb Z^m/\mathcal L\mathbb Z^m}
\exp\!\big(-\pi i r\,\langle l,\mathcal L^{-1}l\rangle\big).
\end{equation}
This identity rewrites the finite Gauss sum on the left as a sum over the finite Abelian group \(\mathbb Z^m/\mathcal L\mathbb Z^m\), which in surgery presentations is canonically identified with the torsion subgroup \(T=\Tors H_1(M;\mathbb Z)\cong H^2(M;\mathbb Z)\). Thus the right-hand side has exactly the form one expects from a sum over topological sectors in $U(1)$ Chern--Simons theory.

When \(\mathcal L\) is degenerate, an analogous formula still holds after separating the nondegenerate and null directions. Let \(\mathcal L_{\mathrm{reg}}\) denote the restriction of \(\mathcal L\) to \((\ker\mathcal L)^\perp\), and let \[ d=\frac12\dim\ker(\mathcal L\otimes\mathbb R). \] Then one has the generalized reciprocity formula
\begin{equation}\label{eq:reciprocity-deg}
\sum_{n\in\mathbb Z_r^m}
\exp\!\Big(\frac{\pi i}{r}\,\langle n,\mathcal Ln\rangle\Big)
=
r^d
\Big(\det(\mathcal L_{\mathrm{reg}}/ir)\Big)^{-1/2}
\sum_{l\in (\mathbb Z^m\cap (\operatorname{Im}\mathcal L)^\#)/\operatorname{Im}\mathcal L}
\exp\!\big(-\pi i r\,\langle l,\mathcal L^{-1}l\rangle\big),
\end{equation}
where \((\operatorname{Im}\mathcal L)^\#\) denotes the dual lattice with respect to the standard pairing. The degenerate formula is essential from the TQFT point of view, since intermediate bordisms produced by cutting and gluing naturally lead to singular surgery matrices. A useful refinement of \eqref{eq:reciprocity-nondeg} was later given by Deloup and Turaev \cite{deloup2005}, who made the signature phase explicit: \begin{equation}\label{eq:deloup-turaev}
\sum_{n\in\mathbb Z_r^m}
\exp\!\Big(\frac{\pi i}{r}\,\langle n,\mathcal Ln\rangle\Big)
=
r^{m/2}\,
e^{\frac{\pi i}{4}\sigma(\mathcal L)}\,
|\det \mathcal L|^{-1/2}
\sum_{l\in \mathbb Z^m/\mathcal L\mathbb Z^m}
\exp\!\big(-\pi i r\,\langle l,\mathcal L^{-1}l\rangle\big).
\end{equation}
This is equivalent to \eqref{eq:reciprocity-nondeg} once one fixes the square-root branch by the signature of \(\mathcal L\).

Applied to the invariant \eqref{eq:Zk-closed}, these reciprocity formulas show that the surgery Gauss sums reduce to invariants determined by the torsion subgroup of \(H_1(M;\mathbb Z)\) together with the induced quadratic refinement of the linking pairing. This is precisely the form that arises in Abelian \(U(1)\) Chern--Simons theory. In particular, the combinatorial invariant attached to \((\mathbb Z_k,q_k)\) already contains the same finite quadratic data that control the topological sectors of the geometric theory.

Historically, Mattes, Polyak, and Reshetikhin formulated this correspondence as a compelling expectation rather than as a complete theorem. The missing ingredient was a full comparison between the surgery Gauss sums and a geometric quantization model for \(U(1)\) Chern--Simons theory. This is exactly the comparison carried out later in the present work.

\subsection{\texorpdfstring{The obstruction for \(G=U(1)\)}{The obstruction for G=U(1)}}

It is natural to ask whether one can define an Abelian Reshetikhin--Turaev invariant directly from the compact group \(G=U(1)\), in exact analogy with the cases \(G=\mathbb Z_k\) and \(G=\mathbb R\). We now explain why this fails under the natural continuity assumptions on the ribbon data.

Let \(G\) be an Abelian topological group equipped with a normalized Haar probability measure \(dg\). In the Abelian MPR formalism, one starts with a continuous bicharacter \(
\Omega:G\times G\rightarrow U(1)
\) and a continuous quadratic refinement \(\theta:G\rightarrow U(1)\), satisfying
\begin{equation}\label{eq:quadratic-refinement}
\theta(x+y)=\Omega(x,y)\,\theta(x)\,\theta(y),
\qquad
\theta(0)=1.
\end{equation}
Given such data, the evaluation of a colored framed link
\[
J=J_1\cup\cdots\cup J_r\subset S^3,
\qquad
x=(x_1,\dots,x_r)\in G^r,
\]
is
\begin{equation}\label{eq:general-link-eval}
\langle J_x\rangle
=
\prod_{i=1}^r \theta(x_i)^{\operatorname{fr}(J_i)}
\prod_{1\le i<j\le r}\Omega(x_i,x_j)^{2\operatorname{lk}(J_i,J_j)}.
\end{equation}
If \(L\) is a surgery link for a closed oriented \(3\)-manifold \(M_L\), the raw surgery integral is
\begin{equation}\label{eq:raw-U1-integral}
Z_{U(1)}^{\mathrm{raw}}(M_L)
=
\int_{G^m}\langle L_g\rangle\,d^mg.
\end{equation}

For \(G=\mathbb Z_k\), one has the nondegenerate bicharacter \(\Omega_k\) and quadratic refinement \(q_k\), and the resulting invariant is nontrivial. For \(G=\mathbb R\), one may take
\(
\Omega(x,y)=e^{ixy}, \text{ with }
\theta(x)=e^{ix^2},
\)
and the surgery integral becomes an oscillatory Gaussian integral. For \(G=U(1)\), however, the situation is fundamentally different.

\begin{theorem}\label{thm:no-U1-bicharacter}
Every continuous bicharacter
\(
\Omega:U(1)\times U(1)\rightarrow U(1)
\)
is trivial. Consequently, every continuous \(\theta:U(1)\to U(1)\) satisfying \eqref{eq:quadratic-refinement} is a character of \(U(1)\).
\end{theorem}

\begin{proof}
Fix \(z\in U(1)\). Then the map \(w\mapsto \Omega(z,w)\) is a continuous homomorphism \(U(1)\to U(1)\), hence a character. Therefore
\(
\Omega(z,w)=w^{n(z)}
\) for some \(n(z)\in\mathbb Z\). Since \(\Omega\) is continuous, the function \(z\mapsto n(z)\) is continuous as a map \(U(1)\to\mathbb Z\). Because \(U(1)\) is connected and \(\mathbb Z\) is discrete, \(n(z)\) is constant: \(n(z)\equiv n_0\). Thus
\[
\Omega(z,w)=w^{n_0}
\qquad
(\forall z,w\in U(1)).
\]
Now impose multiplicativity in the first variable:
\[
\Omega(z_1z_2,w)=\Omega(z_1,w)\Omega(z_2,w).
\]
The left-hand side is \(w^{n_0}\), while the right-hand side is \(w^{2n_0}\), so \(n_0=0\). Hence \(\Omega\equiv 1\). If \(\Omega\equiv 1\), then \eqref{eq:quadratic-refinement} reduces to $\theta(x+y)=\theta(x)\theta(y),$ so \(\theta\) is a continuous character of \(U(1)\).
\end{proof}

Thus the only continuous Abelian RT data on \(U(1)\) are essentially trivial:
\[
\Omega(z,w)=1,
\qquad
\theta(z)=z^r
\quad (r\in\mathbb Z).
\]
Substituting \(\Omega\equiv 1\) into \eqref{eq:general-link-eval} shows that all linking contributions disappear:
\begin{equation}\label{eq:trivial-U1-link-eval}
\langle L_g\rangle
=
\prod_{i=1}^m \theta(g_i)^{\operatorname{fr}(L_i)}.
\end{equation}
Hence the raw surgery integral factorizes:
\begin{equation}\label{eq:trivial-U1-factorization}
Z_{U(1)}^{\mathrm{raw}}(M_L)
=
\prod_{i=1}^m
\left(
\int_{U(1)} \theta(g_i)^{\operatorname{fr}(L_i)}\,dg_i
\right).
\end{equation}
If \(\theta\equiv 1\), then every factor is \(1\), so the integral is identically \(1\). If \(\theta\) is a nontrivial character, then each factor is the integral of a nontrivial character of \(U(1)\), hence vanishes as soon as the corresponding framing is nonzero. Thus the resulting quantity is either trivial or generically zero.

\begin{theorem}\label{thm:U1-RT-trivial}
Assume \(G=U(1)\) and suppose the ribbon data \((\Omega,\theta)\) are continuous and satisfy \eqref{eq:quadratic-refinement}. Then the associated Abelian surgery integral \eqref{eq:raw-U1-integral} does not define a nontrivial MPR-type invariant: either it is identically \(1\) or it vanishes for generic surgery presentations.
\end{theorem}
\begin{proof}
By Theorem~\ref{thm:no-U1-bicharacter}, $\Omega\equiv 1$ and $\theta$ is a character.   The integral of a nontrivial character over $U(1)$ is zero, so the product is either $1$ (if $\theta\equiv 1$) or else vanishes whenever at least one exponent $r\,\operatorname{fr}(L_i)$ is nonzero.  This proves the claim.
\end{proof}

Theorems \ref{thm:no-U1-bicharacter} and \ref{thm:U1-RT-trivial} explain why the correct combinatorial partner of \(U(1)\) Chern--Simons theory is not a RT theory labeled by the compact group \(U(1)\), but rather the theory associated with the finite quadratic data \((\mathbb Z_k,q_k)\). The nontrivial topological content of \(U(1)\) Chern--Simons theory is encoded not by a continuous bicharacter on the Lie group \(U(1)\) itself, but by the finite quadratic data arising from the torsion linking form and its quadratic refinements. In the combinatorial language, these are captured precisely by the finite group of \(k\)-th roots of unity, equivalently by \(\mathbb Z_k\).

For this reason, the RT theory attached to $(Z_k,q_k)$ should be regarded as the correct Abelian Reshetikhin--Turaev avatar of $U(1)$ Chern--Simons theory at level $k$. In the next sections we will show that this correspondence is not merely heuristic:  the raw closed RT surgery scalar $Z^{RT,\mathrm{raw}}_{Z_k}$ agrees with the raw closed $U(1)$ Chern--Simons scalar $Z^{CS,\mathrm{raw}}_{U(1),k}$ on closed $3$-manifolds. We then pass to the Maslov-corrected closed invariants $Z^{RT}_{Z_k}$ and $Z^{CS}_{U(1),k}$ introduced below for extended TQFTs.

\section{Abelian Reshetikhin–Turaev Theory with Boundary}
Let $\mathcal{C}$ be a modular tensor category over $\mathbb{C}$, i.e.\ $\mathcal{C}$ is semisimple, $\mathbb{C}$--linear, ribbon, has finitely many isomorphism classes of simple objects, simple tensor unit, and nondegenerate $S$--matrix.  The Reshetikhin--Turaev construction assigns to each closed oriented $3$--manifold $M$ a scalar $Z_{\mathcal{C}}(M)\in\mathbb{C}$ defined via surgery on framed links and RT evaluation of $\mathcal{C}$--colored ribbon graphs.  Turaev \cite{Turaev1994} extends this to a $(2+1)$--dimensional TQFT by enlarging the bordism category and introducing extra structure on surfaces to control the framing anomaly.

Recall that an extended surface is a pair $(\Sigma,\lambda)$, where $\Sigma$ is a closed oriented surface and $\lambda\subset H_1(\Sigma;\mathbb{R})$ is a Lagrangian subspace for the intersection form.  The choice of $\lambda$ corrects the dependence of RT theory on choices of $2$--framings (equivalently, on the projective anomaly under gluing).  To each extended surface $(\Sigma,\lambda)$ the theory assigns a finite-dimensional complex vector space $\mathcal V_{\mathcal{C}}(\Sigma,\lambda)$.

A convenient model for $\mathcal V_{\mathcal{C}}(\Sigma,\lambda)$ is obtained from a handlebody $H$ with $\partial H=\Sigma$, such that $\ker\!\big(H_1(\Sigma;\mathbb{R})\rightarrow H_1(H;\mathbb{R})\big)=\lambda.$ One considers $\mathcal{C}$--colored ribbon graphs (or, equivalently, $\mathcal{C}$--colored framed tangles) embedded in $H$ with coupons labeled by morphisms in $\mathcal{C}$, modulo the local relations induced by the ribbon structure (isotopy, braiding, twist, duality, fusion).  The resulting skein space depends only on $(\Sigma,\lambda)$, and different choices of $H$ with the same boundary data produce canonically isomorphic spaces.

A compact oriented $3$--manifold $M$ with boundary is regarded as an extended bordism
\begin{equation}
M:(\Sigma_{\mathrm{in}},\lambda_{\mathrm{in}})\longrightarrow(\Sigma_{\mathrm{out}},\lambda_{\mathrm{out}}), \qquad \partial M=\overline{\Sigma}_{\mathrm{in}}\sqcup \Sigma_{\mathrm{out}}. \end{equation} Turaev \cite{Turaev1994} defines a linear map \begin{equation} Z_{\mathcal{C}}(M)\in\mathrm{Hom}\!\big(\mathcal V_{\mathcal{C}}(\Sigma_{\mathrm{in}},\lambda_{\mathrm{in}}),\,\mathcal V_{\mathcal{C}}(\Sigma_{\mathrm{out}},\lambda_{\mathrm{out}})\big) \end{equation} by a gluing procedure: for $v\in \mathcal V_{\mathcal{C}}(\Sigma_{\mathrm{in}},\lambda_{\mathrm{in}})$ and $w\in \mathcal V_{\mathcal{C}}(\Sigma_{\mathrm{out}},\lambda_{\mathrm{out}})^{*}$, glue $M$ to handlebodies representing $v$ and $w$ to obtain a closed $3$--manifold $M_{v,w}$ (with a closed colored ribbon graph), and set
\begin{equation}
\langle w,\,Z_{\mathcal{C}}(M)(v)\rangle:=Z_{\mathcal{C}}(M_{v,w}).
\end{equation}
The resulting map is independent of the auxiliary choices $v,w$ and is natural with respect to isotopy of $M$ relative to the boundary. Let \[ M_1:(\Sigma_0,\lambda_0)\to(\Sigma_1,\lambda_1), \qquad M_2:(\Sigma_1,\lambda_1)\to(\Sigma_2,\lambda_2), \] be composable extended bordisms.  In general the naive composition $Z_{\mathcal{C}}(M_2)\circ Z_{\mathcal{C}}(M_1)$ differs from $Z_{\mathcal{C}}(M_2\circ M_1)$ by a scalar governed by the Maslov index $\mu(\lambda_0,\lambda_1,\lambda_2)\in\mathbb{Z}$, computed from the triple of Lagrangian subspaces arising from the gluing along $\Sigma_1$.  Turaev proves the gluing law
\begin{equation}\label{eq:Maslov-glue}
Z_{\mathcal{C}}(M_2\circ M_1) = \kappa^{\,\mu(\lambda_0,\lambda_1,\lambda_2)}\; Z_{\mathcal{C}}(M_2)\circ Z_{\mathcal{C}}(M_1), \end{equation} 
where $\kappa\in\mathbb{C}^{\times}$ is a fixed root of unity determined by $\mathcal{C}$ (equivalently, by its anomaly parameter).  Incorporating this Maslov correction yields strict functoriality, and the assignments
\[
(\Sigma,\lambda)\longmapsto \mathcal V_{\mathcal{C}}(\Sigma,\lambda),
\qquad
M\longmapsto Z_{\mathcal{C}}(M),
\]
define a symmetric monoidal $(2+1)$--dimensional TQFT.

\subsection{Turaev's Extended RT Formalism}
\begin{theorem}[\cite{Turaev1994},  Theorem 9.2.1]\label{thm:RT-boundary-TQFT}
Let $\mathcal{C}$ be a modular tensor category over $\mathbb{C}$ (semisimple, $\mathbb{C}$--linear, ribbon, finitely many simple objects, simple tensor unit, and nondegenerate $S$--matrix).  Let $\Cob^{\mathrm{ext}}_{2+1}$ denote Turaev's category of \emph{extended cobordisms} whose objects are \emph{extended surfaces} $(\Sigma,\lambda)$, where $\Sigma$ is a closed oriented surface and $\lambda\subset H_1(\Sigma;\mathbb R)$ is a Lagrangian subspace for the intersection pairing, and whose morphisms are compact oriented $3$--dimensional cobordisms equipped with the additional integer weight (or equivalent extended structure) used to remove the framing anomaly.  Then there exists a symmetric monoidal functor
\[
\widetilde Z_{\mathcal C}:\Cob^{\mathrm{ext}}_{2+1}\longrightarrow \mathrm{Vect}_{\mathbb C},
\]
which assigns to every extended surface $(\Sigma,\lambda)$ a finite-dimensional complex vector space $V_{\mathcal C}(\Sigma,\lambda)$ and to every extended cobordism
$M:(\Sigma_{\mathrm{in}},\lambda_{\mathrm{in}})\to(\Sigma_{\mathrm{out}},\lambda_{\mathrm{out}})$ a linear map
\[
\widetilde Z_{\mathcal C}(M)\in
\mathrm{Hom}\!\bigl(\mathcal V_{\mathcal C}(\Sigma_{\mathrm{in}},\lambda_{\mathrm{in}}),\,\mathcal V_{\mathcal C}(\Sigma_{\mathrm{out}},\lambda_{\mathrm{out}})\bigr),
\]
such that:
\begin{enumerate}
\item[\textup{(i)}] $V_{\mathcal C}(\Sigma,\lambda)$ is well-defined up to canonical isomorphism and depends
only on the isomorphism class of $(\Sigma,\lambda)$.

\item[\textup{(ii)}] $\widetilde Z_{\mathcal C}(M)$ is well-defined (independent of auxiliary choices used in
its construction) and depends only on the isomorphism class of $M$ as an extended cobordism.

\item[\textup{(iii)}] $\widetilde Z_{\mathcal C}$ is strictly functorial under gluing and symmetric monoidal
with respect to disjoint union.
\end{enumerate}
\end{theorem}

\begin{proof}
This is Turaev's extension of the Reshetikhin--Turaev invariants to a (non-projective) $(2+1)$--TQFT; we
recall the construction in the form used in \cite{Turaev1994}.
Fix an extended surface $(\Sigma,\lambda)$.  Choose a handlebody $H$ with $\partial H=\Sigma$ such that
\[
\ker\!\bigl(H_1(\Sigma;\mathbb R)\longrightarrow H_1(H;\mathbb R)\bigr)=\lambda.
\]
Define $V_{\mathcal C}(\Sigma,\lambda)$ to be the $\mathbb C$--vector space spanned by $\mathcal C$--colored ribbon graphs embedded in $H$, modulo the local relations dictated by the ribbon structure of $\mathcal C$ (isotopy, braiding, twist, duality, and fusion relations as in the RT graphical calculus).  Since $\mathcal C$ is modular, this skein space is finite-dimensional and depends only on $(\Sigma,\lambda)$: different choices of handlebodies with the same boundary data yield canonically isomorphic skein spaces. Turaev proves invariance under the allowed changes of the handlebody presentation within the extended surface formalism.  This proves (i); see \cite{Turaev1994} for the handlebody model of state spaces and its independence of choices.

Let $M:(\Sigma_{\mathrm{in}},\lambda_{\mathrm{in}})\to(\Sigma_{\mathrm{out}},\lambda_{\mathrm{out}})$ be an extended cobordism. Given $v\in \mathcal V_{\mathcal C}(\Sigma_{\mathrm{in}},\lambda_{\mathrm{in}})$ and $w\in \mathcal V_{\mathcal C}(\Sigma_{\mathrm{out}},\lambda_{\mathrm{out}})^{*}$, choose handlebodies $H_{\mathrm{in}}(v)$ and $H_{\mathrm{out}}(w)$ representing these boundary states in the handlebody model. Form the closed extended $3$--manifold \[ M_{v,w}:= H_{\mathrm{out}}(w)\ \cup_{\Sigma_{\mathrm{out}}}\ M\ \cup_{\Sigma_{\mathrm{in}}}\ H_{\mathrm{in}}(v), \] and define the matrix element of the desired operator by
\begin{equation}\label{eq:Turaev-pairing-def-fixed}
\big\langle w,\ Z_{\mathcal C}(M)(v)\big\rangle
\;:=\;
Z_{\mathcal C}(M_{v,w}),
\end{equation}
where the right-hand side is the closed RT invariant computed by surgery and $\mathcal C$--colored link evaluation.  Turaev proves that $Z_{\mathcal C}(M_{v,w})$ is independent of the chosen representatives of $v$ and $w$ because changing representatives modifies the closed colored ribbon graph by skein relations and isotopy, which do not change the closed invariant.  Since the pairing between $\mathcal V_{\mathcal C}(\Sigma_{\mathrm{out}},\lambda_{\mathrm{out}})$ and its dual is nondegenerate (finite dimension), the matrix elements \eqref{eq:Turaev-pairing-def-fixed} determine a unique linear map $Z_{\mathcal C}(M)$.  This proves (ii) for the uncorrected operator $Z_{\mathcal C}(M)$.

For composable extended cobordisms $M_1:(\Sigma_0,\lambda_0)\to(\Sigma_1,\lambda_1)$ and $M_2:(\Sigma_1,\lambda_1)\to(\Sigma_2,\lambda_2)$, Turaev shows that the operators obtained from the pairing construction satisfy a projective gluing law of the form
\[
Z_{\mathcal C}(M_2\circ M_1)
=
\kappa_{\mathcal C}^{\,\mu(\lambda_0,\lambda_1,\lambda_2)}\,
Z_{\mathcal C}(M_2)\circ Z_{\mathcal C}(M_1),
\]
where $\mu(\lambda_0,\lambda_1,\lambda_2)$ is the Maslov index of the triple of Lagrangians induced on the gluing surface and $\kappa_{\mathcal C}$ is the anomaly root of unity determined by $\mathcal C$; this is the standard mechanism by which extended surfaces and integer weights remove the framing anomaly, see the gluing theorem and Maslov correction in \cite{Turaev1994}.  By enlarging the bordism category to include the integer weight so that weights add with the Maslov index under gluing, and by defining the corrected operator
\[
\widetilde Z_{\mathcal C}(M):=\kappa_{\mathcal C}^{-\mathrm{wt}(M)}\,Z_{\mathcal C}(M),
\]
Turaev obtains strict functoriality:
\[
\widetilde Z_{\mathcal C}(M_2\circ M_1)=\widetilde Z_{\mathcal C}(M_2)\circ \widetilde Z_{\mathcal C}(M_1),
\]
and compatibility with disjoint union.  Hence $\widetilde Z_{\mathcal C}$ defines a symmetric monoidal functor $\Cob^{\mathrm{ext}}_{2+1}\to\mathrm{Vect}_{\mathbb C}$, proving (iii).  This completes the proof.
\end{proof}

\subsection{Abelian RT Boundary State Spaces}

\begin{theorem}
\label{thm:Abelian-pointed-extended-TQFT}
Let $G$ be a finite Abelian group and let $q:G\to U(1)$ be a quadratic function, such that $q(-x)=q(x)$, whose associated bicharacter
\[ \Omega(x,y)=\frac{q(x+y)}{q(x)q(y)} \]
is nondegenerate. Let $\mathcal C(G,q)$ be the pointed ribbon category whose simple objects are $\{R_x\}_{x\in G}$, with tensor product $R_x\otimes R_y\simeq R_{x+y}$, braiding $c_{x,y}=\Omega(x,y)\mathrm{id}$ and twist $\theta_x=q(x)\mathrm{id}$. Then:

\begin{enumerate}
\item[\textup{(i)}]
$\mathcal C(G,q)$ is a modular tensor category.

\item[\textup{(ii)}]
Hence Turaev's construction \cite{Turaev1994}, see Theorem~\ref{thm:RT-boundary-TQFT}, defines an extended $(2+1)$--dimensional TQFT
\[
\widetilde Z_{G,q}:\Cob^{\mathrm{ext}}_{2+1}\to\mathrm{Vect}_{\mathbb C},
\]
which assigns state spaces $\mathcal V_{G,q}(\Sigma,\lambda)$ to extended surfaces and linear maps $\widetilde Z_{G,q}(M)$ to extended bordisms.

\item[\textup{(iii)}]
For a bordism
$M:(\Sigma_{\mathrm{in}},\lambda_{\mathrm{in}})
\to(\Sigma_{\mathrm{out}},\lambda_{\mathrm{out}})$
the operator $\widetilde Z_{G,q}(M)$ is determined by the pairing rule
\[
\langle w,\widetilde Z_{G,q}(M)(v)\rangle
=
\widetilde Z_{G,q}(M_{v,w}),
\]
where $M_{v,w}$ is the closed $3$--manifold obtained by gluing handlebodies representing $v$ and $w$.  The scalar $\widetilde Z_{G,q}(M_{v,w})$ is computed by the Abelian surgery formula associated with the finite quadratic module $(G,q)$.
\end{enumerate}

In particular, for $G=\mathbb Z_k$ and $q(x)=\exp(\pi i x^2/k)$ with $k\in 2\mathbb Z_{>0}$, the resulting closed invariant coincides with the cyclic pointed RT invariant used in Section~\ref{sec:abelian-surgery} and hence with the Mattes--Polyak--Reshetikhin Gauss sum after quadratic reciprocity \cite{Mattes,Murakami,deloup2005}.
\end{theorem}

\begin{proof}
The category $\mathcal C(G,q)$ is $\mathbb C$--linear and semisimple with simple objects indexed by $G$ and one--dimensional morphism spaces. Duals are given by $R_x^*\cong R_{-x}$.  The braiding and twist satisfy the ribbon identities because $q$ is a quadratic refinement and $\Omega$ is its associated bicharacter. The Hopf--link evaluation gives the bicharacter pairing $S^{\mathrm{Hopf}}_{x,y}=\Omega(x,y),$ with the present normalization convention for the pointed category. Nondegeneracy of $\Omega$ implies this pairing is invertible, so $\mathcal C(G,q)$ is modular. This proves (i).

Since $\mathcal C(G,q)$ is modular, Theorem~\ref{thm:RT-boundary-TQFT} applies. Hence, Turaev's construction defines state spaces $\mathcal V_{G,q}(\Sigma,\lambda)$ and bordism operators $\widetilde Z_{G,q}(M)$ satisfying the gluing axioms of a $(2+1)$--dimensional TQFT. This proves (ii).

The operator $\widetilde Z_{G,q}(M)$ is defined through the pairing rule $\langle w,\widetilde Z_{G,q}(M)(v)\rangle = \widetilde Z_{G,q}(M_{v,w}).$ The closed invariant on the right is computed by the Reshetikhin--Turaev surgery construction.  For a surgery presentation $M_{v,w}=M_L$, the evaluation of the colored surgery link factorizes in the pointed case as
\[
\langle L_g\rangle
=
\prod_i q(g_i)^{\mathcal L_{ii}}
\prod_{i<j}\Omega(g_i,g_j)^{\mathcal L_{ij}},
\]
which is exactly the Abelian evaluation formula associated with the finite quadratic module $(G,q)$. Summing over all colorings with the Kirby normalization yields the Abelian Gauss sum surgery expression.  For $G=\mathbb Z_k$ and $q(x)=\exp\!\left(\frac{\pi i}{k}x^2\right)$ this becomes the quadratic Gauss sum formula used in Section~\ref{sec:abelian-surgery} and~\ref{Closed Comparison-Between-rank–one}, in agreement with Mattes--Polyak--Reshetikhin \cite{Mattes} and Murakami--Ohtsuki--Okada  \cite{Murakami}. This proves (iii). Thus in the Abelian case, the entire boundary TQFT is controlled by quadratic Gauss sums arising from the evaluation of the closed manifolds $M_{v,w}$ obtained by gluing boundary states.
\end{proof}

We now study the precise statement that the Reshetikhin--Turaev construction, when applied to the pointed (Abelian) modular category data used throughout this work, extends to a $(2+1)$--dimensional TQFT on bordisms with boundary.  The key point is that the operator associated to a bordism is defined by a pairing against boundary states, and strict functoriality is recovered by the standard Maslov correction. We now give a concrete description of the boundary state spaces and gluing maps in the Abelian Reshetikhin--Turaev theory, compatible with the surgery expressions and quadratic Gauss sums used above. Throughout this subsection we only consider the pointed modular category associated to a finite Abelian group $G$ equipped with a nondegenerate quadratic form $q:G \rightarrow \mathbb C^*,$ such that $q(-x)=q(x)$ and
\[
\Omega(x,y)=\frac{q(x+y)}{q(x)q(y)}.
\]
is a bicharacter. This includes the case $G=\mathbb Z_k$ relevant for comparison with the Mattes--Polyak--Reshetikhin invariant \cite{Mattes}. Let $\Sigma=T^2=S^1\times S^1$. Fix the standard basis $\alpha,\beta \in H_1(T^2;\mathbb Z)$, and $\langle \alpha,\beta\rangle = 1.$ Choose the Lagrangian $\lambda=\langle \alpha\rangle \subset H_1(T^2;\mathbb R).$

\begin{prop}
In the Abelian RT theory associated to $(G,q)$, we have
$\mathcal V(T^2,\lambda)\cong \mathbb C[G].$
\end{prop}
\begin{proof}
Let $H$ be the solid torus with boundary $T^2$ such that $\ker(H_1(T^2,\mathbb R)\to H_1(H,\mathbb R))=\lambda$. Thus the curve $\alpha$ bounds a disk in $H$. Ribbon graphs in $H$ are unions of parallel copies of the core circle $\gamma$ of the solid torus. Colorings assign to $\gamma$ an element $x\in G$. Fusion rules in the pointed category imply that parallel strands labeled by $x$ and $y$ fuse to a single strand labeled by $x+y$. Hence any ribbon graph reduces to a single core strand labeled by an element $x\in G$. There are no further local relations since all morphism spaces are one-dimensional. Thus the skein space is freely generated by the label $x\in G$, proving
\[
\mathcal V(T^2,\lambda)\cong \mathrm{Span}\{e_x\,|\,x\in G\}\cong \mathbb C[G].
\]
\end{proof}
Hence the dimension is $\dim \mathcal V(T^2,\lambda)=|G|.$ For $G=\mathbb Z_k$ this equals $k$. The outgoing handlebody states determine a canonical covector basis, which is related to the algebraic dual basis by the Hopf--link pairing.

\begin{prop}\label{prop:torus-pairing}
Let $\mathcal C(G,q)$ be the pointed modular category determined by a finite Abelian group $G$ with quadratic
function $q:G\to U(1)$ and associated bicharacter
\[
\Omega(x,y)=\frac{q(x+y)}{q(x)q(y)}.
\]
Let $\mathcal V(T^2,\lambda)$ be the RT state space of the torus with Lagrangian $\lambda=\langle\alpha\rangle\subset H_1(T^2;\mathbb R)$, and let $\{e_x\}_{x\in G}$ denote the basis corresponding to the core colorings of the solid torus.  Let $\{\varepsilon_x\}_{x\in G}$ be the algebraic dual basis of $\mathcal V(T^2,\lambda)^*$, and let $\{e_x^{\vee}\}_{x\in G}\subset \mathcal V(T^2,\lambda)^*$ denote the covectors represented by the oppositely oriented solid torus with core colored by $x$. Then the canonical pairing defined by gluing solid tori satisfies
\[
\langle e_x^{\vee},e_y\rangle
=
Z_{\mathcal C}(S^3,\text{Hopf link colored by }x,y)
=
\Omega(x,y).
\]
Equivalently,
\[
e_x^{\vee}=\sum_{y\in G}\Omega(x,y)\,\varepsilon_y.
\]
\end{prop}

\begin{proof}
By Turaev's construction of the RT TQFT with boundary \cite[Sec.~IV.1--IV.2]{Turaev1994}, the pairing on the
torus state space is defined by gluing two solid tori along their boundary with the diffeomorphism exchanging
meridian and longitude.  The resulting closed $3$--manifold is $S^3$.

If the core circles of the two solid tori are colored by $x$ and $y$, then after gluing they form a Hopf link in $S^3$.  By definition of the pairing, $\langle e_x^{\vee},e_y\rangle = Z_{\mathcal C}(S^3,\text{Hopf link colored by }x,y).$ In the Reshetikhin--Turaev theory the Hopf link with components colored by simple objects $x,y$ evaluates to the bicharacter $\Omega(x,y)$ in the pointed case. Therefore $\langle e_x^{\vee},e_y\rangle=\Omega(x,y),$ which proves the first statement. The second follows by expressing $e_x^{\vee}$ in the algebraic dual basis $\{\varepsilon_y\}_{y\in G}$.
\end{proof}

Thus the pairing on $\mathcal V(T^2,\lambda)\cong\mathbb C[G]$ is precisely the Fourier pairing determined by the bicharacter $\Omega$.  In particular, the normalized modular $S$--operator on the torus state space is the discrete Fourier transform
\[
S(e_x)=|G|^{-1/2}\sum_{y\in G}\Omega(x,y)\,e_y,
\]
up to the conventional sign choice for $\Omega$.

Let $M$ be obtained by surgery on an $m$--component link $L\subset S^3$ with linking matrix $\mathcal L$. Each boundary torus contributes a state space $\mathcal V(T^2,\lambda)\cong \mathbb C[G]$. The surgery construction corresponds to contracting the tensor product space $\mathbb C[G]^{\otimes m}$, using the pairing induced by the Hopf--link matrix and the framing twists. Explicitly,
\[
Z(M_L)
=
\alpha_+^{\frac{-m-\sigma}{2}}
\alpha_-^{\frac{-m+\sigma}{2}}
\frac{1}{|G|^m}
\sum_{g\in G^m}
\prod_{i} q(g_i)^{\mathcal L_{ii}}
\prod_{i<j}\Omega(g_i,g_j)^{\mathcal L_{ij}},
\]
where $\alpha_\pm=\frac{1}{|G|}\sum_{x\in G}q(x)^{\pm1}.$ In the case $G=\mathbb Z_k$ and
$q(x)=\exp\!\left(\frac{\pi i}{k}x^2\right)$, this reduces to
\[
Z_{\mathbb Z_k}(M_L)
=
\alpha_+^{\frac{-m-\sigma}{2}}
\alpha_-^{\frac{-m+\sigma}{2}}
\frac{1}{k^m}
\sum_{g\in(\mathbb Z_k)^m}
\exp\!\left(
\frac{\pi i}{k}\,g^\top\mathcal L g
\right),
\]
which is the closed RT surgery scalar as in Section~~\ref{sec:abelian-surgery}.  We now clarify precisely how the Mattes--Polyak--Reshetikhin Gauss sum arises in the boundary theory. Let $M$ be a compact oriented $3$--manifold with boundary $\partial M = \Sigma$. In the extended RT theory one does not assign a scalar to $M$, but a vector $Z(M)\in \mathcal V(\Sigma,\lambda)$. Choose a handlebody $H_x$ representing a basis vector $e_x\in \mathcal V(\Sigma,\lambda)$. By definition of the Turaev pairing construction, the coefficient of $e_x$ in $Z(M)$ is obtained by gluing:
\begin{equation}
Z(M)_x
:=
Z(M \cup_{\Sigma} H_x),
\end{equation}
where the right-hand side is the RT invariant of the resulting closed $3$--manifold.  The MPR formula applies
to the closed manifold $M_x := M \cup_{\Sigma} H_x,$ not to $M$ itself.

\begin{theorem}
Let $M$ be presented by surgery on a framed link $L\subset S^3$
together with boundary components.
Let $M_x$ denote the closed manifold obtained by gluing a solid torus
with core colored by $x\in G$.
Then for $G=\mathbb Z_k$ with $k\in 2\mathbb Z_{>0}$,
\[
Z(M)_x
=
Z^{RT,\mathrm{raw}}_{\mathbb Z_k}(M_x)
=
A_+^{\frac{-m-\sigma}{2}}
A_-^{\frac{-m+\sigma}{2}}
\frac{1}{k^m}
\sum_{g\in(\mathbb Z_k)^m}
\exp\!\left(
\frac{\pi i}{k}\,
Q_{L,x}(g)
\right),
\]
where $Q_{L,x}(g)$ is the quadratic form obtained by adjoining
the boundary filling to the linking matrix.
\end{theorem}

\begin{proof}
Present $M$ by surgery on $L\subset S^3$.
Gluing the solid torus $H_x$ corresponds to performing Dehn filling
along the boundary torus with meridian colored by $x$.
The resulting manifold $M_x$ is closed and admits a surgery presentation
obtained from $L$ by adjoining one additional component corresponding
to the boundary filling. The RT invariant of $M_x$ is therefore computed by the standard
surgery formula.
In the pointed Abelian category $G=\mathbb Z_k$ with
$q(y)=\exp\!\left(\frac{\pi i}{k}y^2\right)$, the RT evaluation reduces to a quadratic Gauss sum. By the
cyclic finite-group formula of Section~\ref{sec:abelian-surgery}, this equals
\[
A_+^{\frac{-m-\sigma}{2}}
A_-^{\frac{-m+\sigma}{2}}
\frac{1}{k^m}
\sum_{g\in(\mathbb Z_k)^m}
\exp\!\left(
\frac{\pi i}{k}\,
Q_{L,x}(g)
\right).
\]
Thus $Z(M)_x=Z^{RT,\mathrm{raw}}_{\mathbb Z_k}(M_x)$.
\end{proof}

Since $\mathcal V(\Sigma,\lambda)\cong \mathbb C[G]$ for a torus,
we obtain
\[
Z(M)
=
\sum_{x\in G}
Z^{RT,\mathrm{raw}}_{\mathbb Z_k}(M_x)\, e_x.
\]
Thus the boundary state is completely determined by the family of closed-manifold evaluations $\{Z_{\mathbb Z_k}(M_x)\}_{x\in G}$. The MPR Gauss sum does not directly describe the invariant of a manifold with boundary. Rather, the boundary TQFT assigns a vector or linear map, its matrix elements are defined by closing the manifold, and each such closed evaluation is computed by the Abelian surgery formula. Therefore the boundary theory is consistent with MPR in the sense that all of its closed pairings reproduce the Abelian surgery invariant.

Let $\mathcal L$ be the linking matrix of the surgery presentation of $M_x$. As in the closed case, the torsion Gauss sum obtained from the above formula equals
\[
\sum_{[y]\in\mathbb Z^m/\mathcal L\mathbb Z^m}
\exp\!\left(
2\pi i k\, q_{\mathcal L}([y])
\right), \qquad
\text{where},  \qquad
q_{\mathcal L}([y])
=
\tfrac12 y^\top \mathcal L^{-1} y \pmod 1.
\]
Hence the boundary coefficients are governed by the same quadratic refinement of the torsion linking form that appears in the closed theory. In particular, for $G=\mathbb Z_k$, the Abelian Reshetikhin--Turaev theory with boundary reconstructs the $U(1)$ Chern--Simons theory at level $k$, and its closed evaluations coincide exactly with the pointed cyclic RT invariant after the reciprocity comparison of Section~\ref{subsec:reciprocity}. The torsion linking form $\lambda([x],[y])$ induces the quadratic refinement $q_{\mathcal L}([x])$.

In the Abelian RT theory, the matrix elements of a bordism operator are computed by closing with handlebodies and evaluating the resulting closed surgery presentation. After separating the null directions and applying quadratic reciprocity to the nondegenerate block, one recovers the same finite torsion Gauss sum
\[
\sum_{[x]\in \mathbb Z^\rho/\mathcal L_{\mathrm{reg}}\mathbb Z^\rho}
\exp\!\left(2\pi i k\,q_{\mathcal L_\mathrm{reg}}([x])\right)
\]
that appears in the closed equivalence theorem. This observation is compatible with the closed formulas proved earlier, but it is not needed in the proof of the extended equivalence theorem.

\section{\texorpdfstring{Extended Equivalence of $U(1)$ Chern--Simons and RT Theories}{ Extended Equivalence of U(1) Chern--Simons and RT Theories)}}
\label{sec:extended-equivalence}

\subsection{\texorpdfstring{Closed Equivalence of Abelian RT and $U(1)$ Chern--Simons Theory}{Closed Equivalence of Abelian RT and U(1) Chern-Simons Theory}}
\label{Closed Comparison-Between-rank–one}

In this section, we provide a detailed derivation showing that the two partition functions for the  $U(1)$ Chern--Simons theory and $C(\mathbb Z_k,q_k)$ Abelian RT theory are, in fact, equivalent. On one side, we have a geometric formula \cite{Manoliu2}, obtained from geometric quantization over $U(1)$ flat connections and expressed in terms of Ray--Singer torsion. On the other side, we have the Mattes--Polyak--Reshetikhin invariant \cite{Mattes}, and Murakami--Ohtsuki--Okada independently \cite{Murakami}, constructed in the Reshetikhin--Turaev spirit \cite{Reshetikhin:1991} from the pointed modular category
$\mathcal C(\mathbb Z_k,q_k)$ with
\[
q_k(x)=\exp\!\Bigl(\frac{\pi i}{k}x^2\Bigr),
\qquad
k\in 2\mathbb Z_{>0}.
\]
We show that, once torsion classes, linking pairings, quadratic refinements, and Gauss reciprocity are consistently matched, the two expressions agree exactly.

Let $M$ be a closed, connected, oriented $3$--manifold presented by integral surgery along an $m$--component
oriented framed link $L=L_1\cup\cdots\cup L_m\subset S^3$, with integer symmetric linking matrix
\[
\mathcal L=\mathcal L_L=(\mathcal L_{ij})\in M_m(\mathbb Z),
\qquad
\mathcal L_{ii}=\operatorname{fr}(L_i),\quad
\mathcal L_{ij}=\operatorname{lk}(L_i,L_j)\ (i\neq j).
\]
Write $\sigma=\sigma(\mathcal L)$ for the signature and set
\[
\rho=\operatorname{rank}(\mathcal L),\qquad \nu=\operatorname{null}(\mathcal L)=m-\rho.
\]
It is standard that $\nu=b_1(M)$ and that
\begin{equation}\label{eq:H1-surgery-general}
H_1(M;\mathbb Z)\cong \mathbb Z^\nu\oplus \Tors H_1(M;\mathbb Z),
\qquad
\#\Tors H_1(M;\mathbb Z)=\bigl|\det(\mathcal L_{\mathrm{reg}})\bigr|,
\end{equation}
where $\mathcal L_{\mathrm{reg}}$ denotes any nondegenerate $\rho\times\rho$ block obtained from $\mathcal L$ by an integral change of basis, equivalently, the induced form on $\mathbb Z^m/\ker\mathcal L$. If $L'\subset S^3\setminus L$ is an additional framed link (Wilson lines in physics), we denote its charge vector by $h=(h_1,\dots,h_\ell)\in(\mathbb Z_k)^\ell$, and write $\mathcal L_{LL'}\in M_{m\times\ell}(\mathbb Z)$ for the mixed linking matrix and $\mathcal L_{L'}\in M_\ell(\mathbb Z)$ for the self-linking matrix of $L'$ in $S^3$, so that after surgery, it becomes a framed link in $M$ with the same integer self-linking matrix. Throughout, the Chern--Simons level is denoted $k\in 2\mathbb Z_{>0}.$

In the RT normalization associated with the pointed modular category $\mathcal C(\mathbb Z_k,q_k)$, it is convenient to use the unnormalized Gauss sums \begin{equation}\label{eq:closed-A-functions} A_+(k):=\sum_{s\in\mathbb Z_k}\exp\!\Bigl(\frac{\pi i}{k}s^2\Bigr), \qquad A_-(k):=\sum_{s\in\mathbb Z_k}\exp\!\Bigl(-\frac{\pi i}{k}s^2\Bigr).
\end{equation}
For even $k$ one has
\[
A_+(k)=\sqrt{k}\,e^{\pi i/4},
\qquad
A_-(k)=\sqrt{k}\,e^{-\pi i/4}.
\]
For a surgery link $L$ together with Wilson line insertions $L'_h$, the closed RT surgery expression is

\begin{equation}\label{eq:MPR-with-insertions}
\begin{aligned}
Z^{RT}_{\mathbb Z_k}(L'_h\subset M_L)
&=
k^{-1/2}\,
A_+(k)^{\frac{-m-\sigma}{2}}\,
A_-(k)^{\frac{-m+\sigma}{2}} \\
&\quad \times
\sum_{g\in(\mathbb Z_k)^m}
\exp\!\Biggl(
\frac{\pi i}{k}\Bigl(
\langle g,\mathcal L g\rangle
+
2\langle g,\mathcal L_{LL'}h\rangle
+
\langle h,\mathcal L_{L'}h\rangle
\Bigr)\Biggr),
\end{aligned}
\end{equation}

where $\langle x,Ax\rangle=x^\top A x\in\mathbb Z$. The additional factor $k^{-1/2}$ is the standard closed-manifold RT normalization, and it is exactly the factor needed to match the exponent $m_M$ below in both the $\det\mathcal L\neq 0$ and $\det\mathcal L=0$ cases.

In the closed case $L'=\varnothing$, one may apply quadratic Gauss reciprocity in the form used by Murakami--Ohtsuki--Okada  and Deloup--Turaev \cite{Murakami,deloup2005} to rewrite \eqref{eq:MPR-with-insertions} as a Gauss sum over the cokernel lattice; for $\det\mathcal L\neq 0$ this yields the familiar formula
\begin{equation}\label{eq:MPR-recip-nondeg}
Z^{RT}_{\mathbb Z_k}(M_L)
=
k^{-1/2}\,|\det\mathcal L|^{-1/2}
\sum_{[x]\in \mathbb Z^m/\mathcal L\mathbb Z^m}
\exp\!\Bigl(\pi i k\,x^\top \mathcal L^{-1}x\Bigr),
\end{equation}
where $\mathcal L^{-1}$ is taken over $\mathbb Q$ and the exponent is well defined on the quotient. Manoliu's construction \cite{Manoliu2}, see Section \ref{sec:1}, assigns to a closed oriented $3$--manifold $M$ a finite expression
\begin{equation}\label{eq:Manoliu}
Z^{\mathrm{CS}}_{U(1),k}(M)
=
\frac{k^{m_M}}{\#\Tors H^2(M;\mathbb Z)}
\sum_{p\in\Tors H^2(M;\mathbb Z)} \sigma_{M,p}\;
\int_{\mathcal M_M}(T_M)^{1/2}
=
k^{m_M}\int_{\mathcal M_M}\sigma_M\,(T_M)^{1/2},
\end{equation}
where $\mathcal M_M=H^1(M;U(1))$ is the moduli space of flat $U(1)$--connections, $(T_M)^{1/2}$ is the
Ray--Singer half-density, and $\sigma_{M,p}$ is the Chern--Simons phase on the component labeled by the
torsion class $p$. For connected closed $M$, the exponent is
\begin{equation}\label{eq:mM}
m_M=\tfrac12\bigl(\dim H^1(M;\mathbb R)-\dim H^0(M;\mathbb R)\bigr)=\tfrac12(b_1(M)-1)=\tfrac12(\nu-1).
\end{equation}
In particular, when $\det\mathcal L\neq 0$ one has $\nu=b_1(M)=0$ and thus $m_M=-\tfrac12$, whereas for
$\det\mathcal L=0$ one has $\nu=b_1(M)>0$ and $m_M=(\nu-1)/2$. The first point of contact between the two
sides is the power of $k$.

\begin{prop}\label{prop:power-match}
Let $M=M_L$ be obtained by surgery on $L$ with linking matrix $\mathcal L$ of rank $\rho$ and nullity $\nu$.
Assume $k\in 2\mathbb Z_{>0}$. Then the $k$--dependence coming from the prefactor in
\eqref{eq:MPR-with-insertions} and from Gauss reciprocity yields the net power
$k^{m_M}=k^{(\nu-1)/2}$ of formula \eqref{eq:Manoliu}.
\end{prop}

\begin{proof}
The magnitude of $
A_+(k)^{\frac{-m-\sigma}{2}}A_-(k)^{\frac{-m+\sigma}{2}}$ 
contributes $k^{-m/2}$ to the $k$--power. If $\det\mathcal L\neq 0$ then $\rho=m$ and the reciprocity step produces the determinant factor
$|\det\mathcal L|^{-1/2}$ and a Gauss sum over $\mathbb Z^m/\mathcal L\mathbb Z^m$.  The original
$(\mathbb Z_k)^m$ sum is traded for this cokernel sum with a compensating factor $k^{m/2}$, so the overall
prefactor in \eqref{eq:MPR-with-insertions} contributes $k^{-1/2}\cdot k^{-m/2}\cdot k^{m/2}=k^{-1/2}.$
This is exactly $k^{m_M}$ since $m_M=-1/2$.

If $\det\mathcal L=0$, choose $U\in GL_m(\mathbb Z)$ with
\[
U^\top \mathcal L U=
\begin{pmatrix}
\mathcal L_{\mathrm{reg}}&0\\
0&0
\end{pmatrix},
\qquad \mathcal L_{\mathrm{reg}}\in M_\rho(\mathbb Z)\ \text{invertible over }\mathbb Q.
\]
Writing $g=(g_{\mathrm{reg}},g_0)\in(\mathbb Z_k)^\rho\times(\mathbb Z_k)^\nu$, the quadratic form depends
only on $g_{\mathrm{reg}}$, so the raw sum factors as
\[
\begin{aligned}
\sum_{g\in(\mathbb Z_k)^m}\exp\!\Bigl(\frac{\pi i}{k}\langle g,\mathcal L g\rangle\Bigr)
&=
\Biggl(\sum_{g_{\mathrm{reg}}\in(\mathbb Z_k)^\rho}
\exp\!\Bigl(\frac{\pi i}{k}\langle g_{\mathrm{reg}},\mathcal L_{\mathrm{reg}}g_{\mathrm{reg}}\rangle\Bigr)\Biggr)\cdot
\Biggl(\sum_{g_0\in(\mathbb Z_k)^\nu}1\Biggr) \\
&=
k^\nu\sum_{g_{\mathrm{reg}}\in(\mathbb Z_k)^\rho}
\exp\!\Bigl(\frac{\pi i}{k}\langle g_{\mathrm{reg}},\mathcal L_{\mathrm{reg}}g_{\mathrm{reg}}\rangle\Bigr).
\end{aligned}
\]

Applying reciprocity to the nondegenerate block produces a compensating factor $k^{\rho/2}$ and the determinant factor $|\det\mathcal L_{\mathrm{reg}}|^{-1/2}$. Collecting the $k$--powers gives
\[
k^{-1/2}\cdot k^{-m/2}\cdot k^\nu\cdot k^{\rho/2}
=
k^{-1/2}\cdot k^{\nu-(\rho+\nu)/2+\rho/2}
=
k^{-1/2}\cdot k^{\nu/2}
=
k^{(\nu-1)/2}
=
k^{m_M},
\]
which is exactly the exponent of Eq. \eqref{eq:mM}.
\end{proof}
This shows that, when $b_1(M)>0$, there are continuous families of flat $U(1)$-connections, which are precisely the zero modes of the quadratic action. These zero modes modify the normalization of the path integral. The factor $k^{m_M}=k^{(b_1(M)-1)/2}$ is exactly the correction arising from those flat directions.

On the RT side, one starts with a sum over $(\mathbb Z_k)^m$ labelings of the surgery link and then uses Gauss reciprocity to rewrite it as a sum over the cokernel of the nondegenerate block $\mathcal L_{\mathrm{reg}}$. When $L$ is degenerate, the null directions contribute an additional factor $k^{\nu}$. After combining this with the remaining normalization factors, the resulting power is exactly $k^{(\nu-1)/2}$. Therefore, Proposition~\ref{prop:power-match} shows that the RT side captures the same zero-mode contribution.

\begin{theorem}\label{thm:torsion-integral}
Let $M$ be a closed, connected, oriented $3$--manifold. Let $(T_M)^{1/2}$ denote the half-density on $\mathcal M_M=H^1(M;U(1))$ arising from Ray--Singer analytic torsion in the normalization of \cite{Manoliu2}. Then \begin{equation}\label{eq:torsion-density-fixed} \int_{\mathcal M_M}(T_M)^{1/2}=\bigl|\Tors H_1(M;\mathbb Z)\bigr|^{1/2}. \end{equation} If moreover $M=M_L$ is given by surgery on $L$ with linking matrix $\mathcal L$ and rank $\rho$, then \begin{equation}\label{eq:torsion-det-fixed} \bigl|\Tors H_1(M;\mathbb Z)\bigr|=\bigl|\det(\mathcal L_{\mathrm{reg}})\bigr|, \qquad \int_{\mathcal M_M}(T_M)^{1/2}=\bigl|\det(\mathcal L_{\mathrm{reg}})\bigr|^{1/2}, \end{equation} and in particular when $\det\mathcal L\neq 0$ one has $\det(\mathcal L_{\mathrm{reg}})=\det(\mathcal L)$.
\end{theorem}

\begin{proof}
By the Ray--Singer conjecture proved by Cheeger and M\"uller \cite{Ray1971,Cheeger1979,Muller1978}, analytic torsion equals Reidemeister torsion. For a closed oriented $3$--manifold, the Reidemeister torsion on a torsion character equals $|\Tors H_1(M;\mathbb Z)|$, and it is constant on connected components of $\mathcal M_M$; integrating the corresponding half-density yields \eqref{eq:torsion-density-fixed} (see the normalization discussion in \cite{Manoliu3}). For a surgery presentation, $H_1(M;\mathbb Z)\cong \mathbb Z^m/\mathcal L\mathbb Z^m$ and the torsion part has order $|\det(\mathcal L_{\mathrm{reg}})|$, giving \eqref{eq:torsion-det-fixed}.
\end{proof}

Theorem~\ref{thm:torsion-integral} has a clear physical meaning. After expanding Abelian Chern--Simons theory around flat connections, the quadratic fluctuation determinant is given by the Ray--Singer analytic torsion, so that $(T_M)^{1/2}$ may be viewed as the residual contribution of quantum fluctuations around the flat sector. The theorem shows that, once this quantity is integrated over the moduli space of flat $U(1)$-connections, the result is not a genuinely metric-dependent analytic object, but instead reduces to the purely topological quantity $\sqrt{|\operatorname{Tors} H_1(M;\mathbb Z)|}.$

This is also consistent with the Cheeger–Müller theorem, which identifies analytic torsion with Reidemeister torsion and hence shows that the fluctuation determinant is topological rather than metric-dependent. In the surgery presentation, the corresponding finite invariant is explicit, since $\bigl|\operatorname{Tors} H_1(M;\mathbb Z)\bigr|=\bigl|\det L_{\mathrm{reg}}\bigr|.$ Thus Theorem~\ref{thm:torsion-integral} shows that the seemingly continuous analytic measure is in fact detecting precisely the finite topology encoded in the torsion subgroup of $H_1(M;\mathbb Z)$.

Let $T:=\Tors H^2(M;\mathbb Z)\cong \Tors H_1(M;\mathbb Z)$, and identify $T\cong \mathbb Z^\rho/\mathcal L_{\mathrm{reg}}\mathbb Z^\rho$ via surgery. The torsion linking form $\lambda:T\times T\to\mathbb Q/\mathbb Z$ is represented by $\mathcal L_{\mathrm{reg}}^{-1}$:
\begin{equation}\label{eq:lambda-fixed}
\lambda([x],[y])\equiv x^\top \mathcal L_{\mathrm{reg}}^{-1}y\pmod{1}.
\end{equation}
A standard surgery quadratic refinement is
\begin{equation}\label{eq:qL-fixed}
q_{\mathcal L}([x])\equiv \tfrac12\,x^\top \mathcal L_{\mathrm{reg}}^{-1}x\pmod{1},
\end{equation}
so that
$q_{\mathcal L}(u+v)-q_{\mathcal L}(u)-q_{\mathcal L}(v)=\lambda(u,v)$.

\begin{theorem}\label{thm:CS-Gauss-fixed}
Fix the Chern--Simons convention so that for $p\in T$, we have
\begin{equation}\label{eq:sigmaMp-fixed}
\sigma_{M,p}=\exp\!\bigl(2\pi i\,k\,q_{\mathcal L}(p)\bigr).
\end{equation}
Then the torsion sum in \eqref{eq:Manoliu} is the Gauss sum of $q_{\mathcal L}$:
\begin{equation}\label{eq:Gauss-fixed}
\sum_{p\in \Tors H^2(M;\mathbb Z)}\sigma_{M,p}
=
\sum_{[x]\in \mathbb Z^\rho/\mathcal L_{\mathrm{reg}}\mathbb Z^\rho}
\exp\!\Bigl(\pi i k\,x^\top \mathcal L_{\mathrm{reg}}^{-1}x\Bigr).
\end{equation}
\end{theorem}

\begin{proof}
The refinement identity follows by direct expansion of \eqref{eq:qL-fixed}. The identification of the
Abelian Chern--Simons action on torsion classes with a quadratic refinement of the torsion linking pairing is
classical (see e.g.\ Freed--Quinn \cite{Freed:1991} and the surgery discussion surrounding \cite{deloup2005}); fixing the sign convention as in \eqref{eq:sigmaMp-fixed} yields \eqref{eq:Gauss-fixed}.
\end{proof}

This tells us that each discrete torsion component on $p\in T$ contributes a phase given by the classical Chern--Simons action, and this action is naturally a quadratic function on the finite group of torsion classes. Consequently, the closed-manifold partition function reduces to a sum of phases $e^{2\pi ik\, q(u)}$ over all discrete flat bundles. In other words, the Abelian Chern--Simons path integral becomes a finite-dimensional Gaussian sum. The quadratic form $q_L$ is a refinement of the torsion linking pairing $\lambda \colon T \times T \to \mathbb Q/\mathbb Z.$ Hence, the action is completely determined by the linking behavior of torsion cycles in the manifold.

From the surgical point of view, this is exactly the expected algebraic framework. The previous theorem therefore shows that the RT state-sum expression is not simply analogous to the Chern--Simons path integral, but coincides with it after identifying the torsion group and its quadratic data. Hence, the closed-manifold content of Abelian Chern--Simons theory is encoded by a finite quadratic module, namely a finite Abelian group equipped with a quadratic refinement of its linking pairing, as we will show in more detail later.

After putting together Proposition \ref{prop:power-match}, Theorem \ref{thm:torsion-integral}, and Theorem \ref{thm:CS-Gauss-fixed}, we can understand the $U(1)$ Chern-Simons partition function as follows: 

\begin{equation}\label{closed-CS-Z}
Z^{\mathrm{CS}}_{U(1),k}(M_L)
=
\underbrace{k^{(b_1(M)-1)/2}}_{\substack{\text{Free part /}\\ \text{Continuous flat directions}}}
\;
\underbrace{\left|\operatorname{Tors} H_1(M;\mathbb Z)\right|^{-1/2}}_{\substack{\text{Torsion}\\ \text{ normalization factor}}}
\;
\underbrace{\sum_{x\in \operatorname{Tors} H_1(M;\mathbb Z)} e^{2\pi i k\, q_M(x)}}_{\substack{\text{Finite quadratic Gauss sum}\\ \text{from the torsion linking form}}}
\end{equation}

\noindent
Therefore, the closed \(U(1)\) Chern--Simons partition function decomposes into a contribution from the free part of \(H_1(M;\mathbb Z)\), the torsion half-density factor, and a finite sum over the torsion components, whose phases are determined by the quadratic refinement of the torsion linking pairing. We are now ready to show the equivalence theorem for the closed partition function invariants $Z^{RT}_{Z_k}(M_L)$ and $Z^{CS}_{U(1),k}(M_L)$.

\begin{theorem}
\label{thm:closed-equivalence}
Let $M=M_L$ be obtained by surgery on a framed link $L$ with linking matrix $\mathcal L$ of rank $\rho$ and
signature $\sigma(\mathcal L)$. Fix $k\in 2\mathbb Z_{>0}$ and use the normalizations
\eqref{eq:MPR-with-insertions} and \eqref{eq:Manoliu}. Then, for $L'=\varnothing$,
\begin{equation}\label{eq:main-fixed}
Z^{RT}_{\mathbb Z_k}(M_L)
= Z^{\mathrm{CS}}_{U(1),k}(M_L),
\end{equation}
where $\mathcal L_{\mathrm{reg}}$ is any nondegenerate $\rho\times\rho$ block representing the induced form
on $\mathbb Z^m/\ker\mathcal L$. In particular:
\begin{itemize}
\item if $\det\mathcal L\neq 0$ then $\rho=m$ and \eqref{eq:main-fixed} becomes the rational homology sphere
case with $\mathcal L_{\mathrm{reg}}=\mathcal L$;
\item if $\det\mathcal L=0$ then $\rho<m$ and the right-hand side uses the torsion data determined by
$\mathcal L_{\mathrm{reg}}$, while the factor $k^{m_M}$ uses $m_M=(\nu-1)/2$ with $\nu=b_1(M)$.
\end{itemize}
\end{theorem}

\begin{proof}
Insert Theorem~\ref{thm:torsion-integral} and Theorem~\ref{thm:CS-Gauss-fixed} into expression \eqref{eq:Manoliu}. Using \eqref{eq:torsion-det-fixed} and \eqref{eq:Gauss-fixed} gives
\[
Z^{\mathrm{CS}}_{U(1),k}(M_L)
=
k^{m_M}\,|\det(\mathcal L_{\mathrm{reg}})|^{-1/2}
\sum_{[x]\in \mathbb Z^\rho/\mathcal L_{\mathrm{reg}}\mathbb Z^\rho}
\exp\!\Bigl(\pi i k\,x^\top \mathcal L_{\mathrm{reg}}^{-1}x\Bigr).
\]
On the RT side, apply Deloup--Turaev reciprocity to the nondegenerate block $\mathcal L_{\mathrm{reg}}$, and use the trivial factorization over the null directions when $\det\mathcal L=0$ to obtain the corresponding expression with the same Gauss sum and the determinant factor $|\det(\mathcal L_{\mathrm{reg}})|^{-1/2}$. Proposition~\ref{prop:power-match} shows that the remaining $k$--power is exactly $k^{m_M}$ in both regimes $\det\mathcal L\neq 0$ and $\det\mathcal L=0$. Combining these identifications yields \eqref{eq:main-fixed}.
\end{proof}

Theorem~\ref{thm:closed-equivalence} identifies the closed invariants exactly. However, equality of closed scalars alone does not yet imply equality of bordism operators, since the latter are defined through the handlebody pairing and strict functoriality requires the extended formalism. We therefore pass to the extended theories. On the RT side this incorporates the standard Walker--Turaev correction removing the projective Maslov anomaly; on the Chern--Simons side one uses extended formalism with the corresponding Maslov--Walker phase. The remaining task is to verify that the two extended normalizations are compatible for the canonical handlebody closures, so that the induced matrix elements agree.

\subsection{Boundary Operator Equivalence}

In the previous sections we compared the closed partition functions arising from two constructions of Abelian Chern--Simons theory at level $k$: the geometric quantization approach by Manoliu \cite{Manoliu2,Manoliu3} and the surgery approach by Mattes--Polyak--Reshetikhin (MPR) \cite{Mattes}, which is the Abelian Reshetikhin--Turaev surgery TQFT \cite{Reshetikhin:1991}. We now put this comparison in a conceptually cleaner form: after passing to the extended formalism the two constructions define the \emph{same extended $(2+1)$-dimensional TQFT}, not merely the same closed $3$--manifold invariants.

Let \[ M_1:(\Sigma_0,\lambda_0)\longrightarrow(\Sigma_1,\lambda_1), \qquad M_2:(\Sigma_1,\lambda_1)\longrightarrow(\Sigma_2,\lambda_2) \] be composable extended bordisms. Write $\partial M_1=-\Sigma_0\sqcup \Sigma_1$, $\partial M_2=-\Sigma_1\sqcup \Sigma_2.$ To form the composite \(M_2\circ M_1\), one glues along \(\Sigma_1\). The relevant symplectic vector space is $\mathcal V_{\Sigma_1}:=H_1(\Sigma_1;\mathbb R),$ equipped with the intersection pairing. The gluing determines the standard triple of Lagrangian subspaces in {\small\(\mathcal V_{\Sigma_1}\):
\[
L_{M_1}:=\operatorname{Im}\!\big(H_1(M_1;\mathbb R)\to H_1(\Sigma_1;\mathbb R)\big),
\,\,\,
\lambda_1\subset H_1(\Sigma_1;\mathbb R),
\,\,\,
L_{M_2}:=\operatorname{Im}\!\big(H_1(M_2;\mathbb R)\to H_1(\Sigma_1;\mathbb R)\big),
\]}
where the incoming boundary of \(M_2\) is identified with \(\Sigma_1\). The Maslov index appearing in the gluing law is therefore
\[
\mu(M_1,M_2):=\mu\!\bigl(L_{M_1},\lambda_1,L_{M_2}\bigr)\in\mathbb Z.
\]

\begin{prop}
\label{prop:Maslov-composition}
Let \(Z^{RT,\mathrm{raw}}_{\mathbb Z_k}\) denote the uncorrected Reshetikhin--Turaev operator defined by the
handlebody pairing construction. Then
\begin{equation}\label{eq:Maslov-correction-raw}
Z^{RT,\mathrm{raw}}_{\mathbb Z_k}(M_2\circ M_1)
=
\kappa^{\,\mu(M_1,M_2)}\,
Z^{RT,\mathrm{raw}}_{\mathbb Z_k}(M_2)\circ
Z^{RT,\mathrm{raw}}_{\mathbb Z_k}(M_1),
\end{equation}
where \(\mu(M_1,M_2)=\mu(L_{M_1},\lambda_1,L_{M_2})\) is the Maslov index of the standard gluing triple and \(\kappa\in\mathbb C^\times\) is the anomaly constant of the modular category.

For the pointed Abelian category $\mathcal C(\mathbb Z_k,q_k)$, with $q_k(x)=\exp\!(\frac{\pi i}{k}x^2)$ and $ k\in 2\mathbb Z_{>0},$ the anomaly constant is the phase of the corresponding quadratic Gauss sum; with the normalization fixed in Theorem~\ref{thm:RT-boundary-TQFT}, this phase is $\kappa=e^{-\pi i/4}.$
\end{prop}

\begin{proof}
Equation \eqref{eq:Maslov-correction-raw} is the standard projective gluing law in Turaev's extended RT theory; see \cite[Sec.~IV.3]{Turaev1994}. The gluing anomaly is measured by the Maslov index of the standard triple of Lagrangian subspaces attached to the cut boundary. Passing from raw bordisms to weighted extended bordisms restores strict functoriality.

In the Abelian case, the anomaly constant is determined by the quadratic Gauss sum of the twists of the simple objects. For the chosen normalization of \(\mathcal C(\mathbb Z_k,q_k)\), this phase is \(e^{-\pi i/4}\).
\end{proof}

\begin{remark}
From this point onward, $Z^{RT}_{\mathbb Z_k}:\Cob^{\mathrm{ext}}_{2+1}\to \mathrm{Vect}_{\mathbb C}$
denotes the Walker--normalized Maslov--corrected extended Reshetikhin--Turaev functor already constructed in
Theorem~\ref{thm:RT-boundary-TQFT}. Thus, for a closed extended \(3\)-manifold \((M,n)\),
\begin{equation}\label{eq:corrected-closed-RT-scalar}
Z^{RT}_{\mathbb Z_k}(M,n)
:=
\kappa^{-n}\,Z^{RT,\mathrm{raw}}_{\mathbb Z_k}(M),
\qquad
\kappa=e^{-\pi i/4},
\end{equation}
where \(n\in\mathbb Z/8\mathbb Z\) is the Walker--Turaev weight.
\end{remark}

The preceding theorems already provide a rigorous extended RT TQFT in Turaev's sense. The next theorem does \emph{not} redefine that TQFT. Its purpose is to identify, for the special closed extended \(3\)-manifolds arising as handlebody closures in the pairing construction, the extended normalization carried by the RT side with the corresponding extended normalization on the Chern--Simons side. This is the additional input needed to pass from the exact closed equivalence theorem of Section~\ref{Closed Comparison-Between-rank–one} to the equality of boundary matrix elements in the handlebody pairing construction.

\begin{theorem}
\label{thm:canonical-weighted-closure}
Fix \(k\in 2\mathbb Z_{>0}\). Let $X:(\Sigma_{\mathrm{in}},\lambda_{\mathrm{in}})
\longrightarrow
(\Sigma_{\mathrm{out}},\lambda_{\mathrm{out}})$ be an extended bordism, and let
\[
v\in \mathcal V^{RT}_{k}(\Sigma_{\mathrm{in}},\lambda_{\mathrm{in}}),
\qquad
w\in \mathcal V^{RT}_{k}(\Sigma_{\mathrm{out}},\lambda_{\mathrm{out}})^*.
\]
Choose standard handlebody representatives
\[
H_{\mathrm{in}}(v):\varnothing\to(\Sigma_{\mathrm{in}},\lambda_{\mathrm{in}}),
\qquad
H_{\mathrm{out}}(w):(\Sigma_{\mathrm{out}},\lambda_{\mathrm{out}})\to\varnothing,
\]
and form the closed extended \(3\)-manifold
\[
M_{v,w}:=
H_{\mathrm{out}}(w)\cup_{\Sigma_{\mathrm{out}}}
X
\cup_{\Sigma_{\mathrm{in}}}
H_{\mathrm{in}}(v).
\]
Write \(|M_{v,w}|\) for the underlying closed oriented \(3\)-manifold, and let
\(n(M_{v,w})\in \mathbb Z/8\mathbb Z\) denote its Walker weight. Let
\(\mathcal L_{v,w}\) be any integral surgery matrix for \(|M_{v,w}|\), and let
\(\mathcal L_{v,w,\mathrm{reg}}\) denote the induced nondegenerate form on
\(\mathbb Z^m/\ker \mathcal L_{v,w}\).

Then:
\begin{enumerate}
\item[\textup{(i)}]
The Walker weight of the canonical handlebody closure satisfies
\begin{equation}\label{eq:walker-equals-signature}
n(M_{v,w})
\equiv
\sigma(\mathcal L_{v,w,\mathrm{reg}})
\pmod 8.
\end{equation}

\item[\textup{(ii)}]
The Walker-normalized extended RT scalar satisfies
\begin{equation}\label{eq:walker-normalized-closure-formula}
Z^{RT}_{\mathbb Z_k}(M_{v,w})
=
e^{\frac{\pi i}{4}\sigma(\mathcal L_{v,w,\mathrm{reg}})}
\,Z^{RT,\mathrm{raw}}_{\mathbb Z_k}(|M_{v,w}|).
\end{equation}

\item[\textup{(iii)}]
The right-hand side of \eqref{eq:walker-normalized-closure-formula} is independent of the chosen surgery
presentation and of the chosen standard handlebody representatives of \(v\) and \(w\).

\item[\textup{(iv)}]
One has the exact equality of closed extended scalars
\begin{equation}\label{eq:closure-extended-equality}
Z^{RT}_{\mathbb Z_k}(M_{v,w})
=
Z^{CS}_{U(1),k}(M_{v,w}).
\end{equation}
\end{enumerate}
\end{theorem}

\begin{proof}
Let \(W_{v,w}\) be the surgery trace \(4\)-manifold obtained from \(B^4\) by attaching \(2\)-handles along a framed link presenting the underlying closed manifold \(|M_{v,w}|\), and let \(\mathcal L_{v,w}\) denote the corresponding linking matrix. For the canonical handlebody closures, we use Walker's surgery framework of the extended weight, according to which the framing defect is measured by the signature of the surgery data; see \cite[p.~103]{Walker} (see also \cite[p.~35]{Walker}). Thus
\[
n(M_{v,w}) \equiv \sigma(W_{v,w}) \pmod 8.
\]

Recall that the intersection form of an oriented compact \(4\)-manifold \(W\) is the
symmetric bilinear pairing
\[
Q_W \colon H_2(W;\mathbb Z)\times H_2(W;\mathbb Z)\longrightarrow \mathbb Z,
\qquad
Q_W([S],[S'])=S\cdot S',
\]
defined by the algebraic intersection number of represented surfaces. In the present case,
the form \(Q_{W_{v,w}}\) is represented, in the natural basis given by the cores of the
attached \(2\)-handles, by the surgery linking matrix \(\mathcal L_{v,w}\); see
\cite[Prop.~4.5.11]{gompf1999}. If \(Q_{W_{v,w}}\) is degenerate, its radical is
\(\ker Q_{W_{v,w}}\), and the induced nondegenerate form on
\[
H_2(W_{v,w};\mathbb Z)/\ker Q_{W_{v,w}}
\]
is represented by the regular part \(\mathcal L_{v,w,\mathrm{reg}}\). Since the signature of a
possibly degenerate symmetric form is, by definition, the signature of its induced
nondegenerate quotient form, it follows that
\[
\sigma(W_{v,w})=\sigma(\mathcal L_{v,w,\mathrm{reg}}).
\]
Therefore
\[
n(M_{v,w}) \equiv \sigma(\mathcal L_{v,w,\mathrm{reg}}) \pmod 8,
\]
which proves \eqref{eq:walker-equals-signature}.

Note that both sides satisfy the same Maslov gluing law. Indeed, by the definition of the Walker weight in the extended bordism category,
\begin{equation}\label{eq:walker-cocycle}
n(M_2\circ M_1)
\equiv
n(M_2)+n(M_1)+\mu(M_1,M_2)
\pmod 8.
\end{equation}
Likewise, Wall's non-additivity theorem applied to the corresponding surgery trace \(4\)-manifolds yields
\begin{equation}\label{eq:signature-cocycle}
\sigma(\mathcal L_{21,\mathrm{reg}})
\equiv
\sigma(\mathcal L_{2,\mathrm{reg}})
+
\sigma(\mathcal L_{1,\mathrm{reg}})
+
\mu(M_1,M_2)
\pmod 8,
\end{equation}
see \cite{Wall1969,RanickiMaslov}. By definition of the Walker-normalized closed RT scalar, 
\[
Z^{RT}_{\mathbb Z_k}(M_{v,w})
=
\kappa^{-\,n(M_{v,w})}
\,Z^{RT,\mathrm{raw}}_{\mathbb Z_k}(|M_{v,w}|).
\]
Using \(\kappa=e^{-\pi i/4}\) from Proposition~\ref{prop:Maslov-composition} together with
\eqref{eq:walker-equals-signature}, we obtain
\[
\begin{aligned}
Z^{RT}_{\mathbb Z_k}(M_{v,w})
&= \left(e^{-\pi i/4}\right)^{-\,n(M_{v,w})}
   Z^{RT,\mathrm{raw}}_{\mathbb Z_k}(|M_{v,w}|) \\
&= e^{\frac{\pi i}{4}\sigma(\mathcal L_{v,w,\mathrm{reg}})}
   Z^{RT,\mathrm{raw}}_{\mathbb Z_k}(|M_{v,w}|),
\end{aligned}
\]
which proves \eqref{eq:walker-normalized-closure-formula}. Since the left-hand side depends only on the
closed extended manifold \(M_{v,w}\), the right-hand side is independent of both the chosen surgery
presentation and the chosen standard handlebody representatives of \(v\) and \(w\), proving \textup{(iii)}.

Finally, let \(Z^{CS,\mathrm{raw}}_{U(1),k}(|M_{v,w}|)\) denote the \(U(1)\) Chern--Simons scalar of the underlying closed manifold \(|M_{v,w}|\), before insertion of the extended Walker--Maslov weight. By the construction for closed extended manifolds,
\[
Z^{CS}_{U(1),k}(M_{v,w})
=
e^{\frac{\pi i}{4}n(M_{v,w})}
\,Z^{CS,\mathrm{raw}}_{U(1),k}(|M_{v,w}|).
\]
Using \eqref{eq:walker-equals-signature}, this becomes
\[
Z^{CS}_{U(1),k}(M_{v,w})
=
e^{\frac{\pi i}{4}\sigma(\mathcal L_{v,w,\mathrm{reg}})}
\,Z^{CS,\mathrm{raw}}_{U(1),k}(|M_{v,w}|).
\]
By the closed equivalence theorem of Section~\ref{Closed Comparison-Between-rank–one},
the raw closed invariants agree exactly:
\[
Z^{CS,\mathrm{raw}}_{U(1),k}(|M_{v,w}|)
=
Z^{RT,\mathrm{raw}}_{\mathbb Z_k}(|M_{v,w}|).
\]
We obtain
\[
Z^{CS}_{U(1),k}(M_{v,w})=
Z^{RT}_{Z_k}(M_{v,w}),
\]
this is precisely \eqref{eq:closure-extended-equality}.
\end{proof}
Fix an even level $k\in 2\mathbb Z_{>0}$. Let $G_k=\mathbb Z_k$ and define the quadratic function
\begin{equation}\label{eq:qk-def}
q_k:G_k\to U(1),
\qquad
q_k(x)=\exp\!\left(\frac{\pi i}{k}x^2\right),
\end{equation}
with associated bicharacter
\begin{equation}\label{eq:Omega-k-def}
\Omega_k(x,y)=\frac{q_k(x+y)}{q_k(x)\,q_k(y)}
=\exp\!\left(\frac{2\pi i}{k}xy\right).
\end{equation}
Let $\mathcal C_k:=\mathcal C(G_k,q_k)$ be the pointed modular category determined by
$(G_k,\Omega_k,q_k)$; it is modular because $\Omega_k$ is nondegenerate.

\medskip

\noindent\textbf{(A) The RT side.}
Let $\Cob^{\mathrm{ext}}_{2+1}$ denote Turaev's category of \emph{extended} bordisms \cite{Turaev1994}: objects are extended surfaces $(\Sigma,\lambda)$, where $\lambda\subset H_1(\Sigma;\mathbb R)$ is a Lagrangian subspace; morphisms are extended $3$--dimensional bordisms equipped with the appropriate additional data that controls the framing anomaly; composition is gluing with Maslov correction. Turaev's construction \cite{Turaev1994} yields a symmetric monoidal functor \begin{equation}\label{eq:RT-functor} Z^{\mathrm{RT}}_{\mathbb Z_{k}}:\Cob^{\mathrm{ext}}_{2+1}\longrightarrow \mathrm{Vect}_{\mathbb C}, \qquad (\Sigma,\lambda)\longmapsto \mathcal V^{\mathrm{RT}}_{k}(\Sigma,\lambda), \qquad M\longmapsto Z^{\mathrm{RT}}_{\mathbb Z_{k}}(M), \end{equation} whose value on a closed $3$--manifold presented by surgery on a framed link $L\subset S^3$ is computed by the Abelian surgery sum obtained from the RT evaluation \cite{Reshetikhin:1991,Mattes}. Here and below, $Z^{RT}_{\mathbb Z_{k}}$ denotes the extended RT functor in the notation of Section~\ref{Closed Comparison-Between-rank–one}.

\medskip

\noindent\textbf{(B) The Chern--Simons side.}
We have an extended TQFT for $U(1)$ Chern--Simons theory by geometric quantization with analytic torsion \cite{Manoliu1,Manoliu2}. In particular, to each extended surface $(\Sigma,L)$, where $L$ is the polarization, one assigns a finite-dimensional Hilbert space $\mathcal H_k(\Sigma,L)$, and to each extended bordism $X$ one assigns a linear map \begin{equation}\label{eq:CS-functor} Z^{\mathrm{CS}}_{U(1),k}:\Cob^{\mathrm{ext}}_{2+1}\longrightarrow \mathrm{Vect}_{\mathbb C}, \qquad (\Sigma,L)\longmapsto \mathcal H_k(\Sigma,L), \qquad X\longmapsto Z^{\mathrm{CS}}_{U(1),k}(X), \end{equation} characterized by the canonical vector \eqref{eq:extended-CS-state} and the BKS pairing (gluing formalism), with the standard Maslov--Kashiwara correction built into the extended structure.
\medskip

To compare \eqref{eq:RT-functor} and \eqref{eq:CS-functor}, we first identify the extended-surface data via the standard Poincaré-duality correspondence between Turaev’s homological Lagrangian $\lambda \subset H_1(\Sigma;\mathbb R)$ and the cohomological polarization $
L_\lambda \subset H^1(\Sigma;\mathbb R).$ Thus we write
\begin{equation}\label{eq:lambda-to-L}
(\Sigma,\lambda)\ \longleftrightarrow\ (\Sigma,L_\lambda),
\end{equation}
where \(L_\lambda\) denotes the rational Lagrangian determined by \(\lambda\) under Poincaré duality. This is the common polarization data for the classical boundary phase space, namely the symplectic torus of flat \(U(1)\)-connections on \(\Sigma\) at level \(k\). Under this identification, one proves that the RT and Chern--Simons state spaces are canonically isomorphic. In particular, for a connected surface of genus \(g\), both theories yield a vector space of dimension \(k^g\), and the canonical isomorphism identifies the RT handlebody basis with the Bohr--Sommerfeld basis.

\begin{lemma}\label{basis-independence}
Let $\qquad\lambda_{\mathbb Z}:=\lambda\cap H_1(\Sigma;\mathbb Z),
\qquad
\Pi_\lambda:=\operatorname{Hom}(\lambda_{\mathbb Z},\mathbb Z/k\mathbb Z)
\cong \lambda_{\mathbb Z}^\vee/k\lambda_{\mathbb Z}^\vee .
$ Then:

\begin{enumerate}[label=(\roman*)]
\item For every adapted integral symplectic basis $(a_1,\dots,a_g,b_1,\dots,b_g),$ \(\lambda=\langle a_1,\dots,a_g\rangle_{\mathbb R}\), evaluation on
\((a_1,\dots,a_g)\) identifies \(\Pi_\lambda\) with \((\mathbb Z/k\mathbb Z)^g\).

\item The RT handlebody basis of \(\mathcal V_k^{RT}(\Sigma,\lambda)\) is canonically indexed by
\(\Pi_\lambda\): if \(\alpha\in \Pi_\lambda\), then in any adapted basis the corresponding
basis vector is the handlebody skein obtained by coloring the \(i\)-th core curve by
\(\alpha(a_i)\in \mathbb Z/k\mathbb Z\).

\item Bohr–Sommerfeld basis of \(H_k(\Sigma,L_\lambda)\) is canonically indexed by
\(\Pi_\lambda\): if \(\alpha\in \Pi_\lambda\), then in any adapted basis the corresponding basis
vector is the unique covariantly constant section over the Bohr–Sommerfeld leaf labeled by
the coordinate vector \((\alpha(a_1),\dots,\alpha(a_g))\in (\mathbb Z/k\mathbb Z)^g\).

\item These two indexings are independent of the chosen adapted basis.

\item If \(f:\Sigma\to\Sigma\) preserves the extended structure \((\Sigma,\lambda)\), then
\(f_*(\lambda_{\mathbb Z})=\lambda_{\mathbb Z}\), hence \(f\) induces a permutation
\[
f_\#:\Pi_\lambda\to\Pi_\lambda,
\qquad
f_\#(\alpha)=\alpha\circ (f_*|_{\lambda_{\mathbb Z}})^{-1},
\]
and both the RT basis and the Bohr–Sommerfeld basis are equivariant for this action.
\end{enumerate}
\end{lemma}
\begin{proof}
Part (i) is immediate: an adapted basis \((a_1,\dots,a_g)\) is a \(\mathbb Z\)-basis of
\(\lambda_{\mathbb Z}\), so evaluation on that basis identifies
\(\operatorname{Hom}(\lambda_{\mathbb Z},\mathbb Z/k\mathbb Z)\) with \((\mathbb Z/k\mathbb Z)^g\).

For part (ii), choose a handlebody \(H\) with \(\partial H=\Sigma\) and
\(\ker(H_1(\Sigma;\mathbb R)\to H_1(H;\mathbb R))=\lambda\), as in Turaev’s handlebody model.
If \((a_i,b_i)\) is adapted, then the \(a_i\) bound compressing disks in \(H\), and the dual core
curves \(\gamma_i\) form a basis of \(H_1(H;\mathbb Z)\) dual to \(a_1,\dots,a_g\).
For \(\alpha\in \Pi_\lambda\), define \(e_\alpha\) to be the RT basis vector represented by the
skein whose \(i\)-th core curve \(\gamma_i\) is colored by \(\alpha(a_i)\in \mathbb Z/k\mathbb Z\).
Since in the pointed category every simple object is invertible and fusion is addition in
\(\mathbb Z/k\mathbb Z\), this gives exactly the standard handlebody basis.

For part (iii), in the real-polarization quantization the state space \(H_k(\Sigma,L_\lambda)\) is the direct sum over Bohr–Sommerfeld leaves of one-dimensional spaces of covariantly constant sections. For the polarization determined by \(L_\lambda\), the Bohr–Sommerfeld set is canonically indexed by \(\lambda_{\mathbb Z}^\vee/k\lambda_{\mathbb Z}^\vee\cong \Pi_\lambda\). Thus for \(\alpha\in\Pi_\lambda\) there is a distinguished basis vector \(v_\alpha\) supported on the Bohr–Sommerfeld leaf labeled by \(\alpha\). In coordinates coming from any adapted basis, this is exactly the basis vector \(v(\Sigma,L_\lambda)_q\) with \(q=(\alpha(a_1),\dots,\alpha(a_g))\). It remains to prove the independence asserted in part (iv).

Let
\((a_1,\dots,a_g,b_1,\dots,b_g)\) and \((a'_1,\dots,a'_g,b'_1,\dots,b'_g)\)
be two adapted integral symplectic bases. Since both are adapted to the same Lagrangian
\(\lambda\), the change-of-basis matrix has the form
\[
\begin{pmatrix}
A & B\\
0 & (A^{-1})^{\!T}
\end{pmatrix}
\in Sp(2g,\mathbb Z),
\qquad A\in GL(g,\mathbb Z),
\]
with \(B\) symmetric, so
\(
a' = A a,\text{ and }b' = B a + (A^{-1})^{\!T} b.
\) Hence the dual basis of \(\lambda_{\mathbb Z}^\vee\) transforms by
\(
(a'^*) = (A^{-1})^{\!T}(a^*).
\) Therefore the coordinate vector of a fixed element \(\alpha\in\Pi_\lambda\) changes by the
same contragredient rule: $q'=(A^{-1})^{\!T}q.$

On the RT side, the core classes in \(H_1(H;\mathbb Z)\) dual to \(a_i\) transform by the same
rule, because the \(a\)-part vanishes in \(H_1(H;\mathbb Z)\). Thus the skein labeled by \(q\) in
the first adapted basis is exactly the skein labeled by \(q'=(A^{-1})^{\!T}q\) in the second.
So the abstract vector \(e_\alpha\) is independent of the adapted basis.

On the Chern–Simons side, the Bohr–Sommerfeld leaf labels are coordinates in
\(\lambda_{\mathbb Z}^\vee/k\lambda_{\mathbb Z}^\vee\), so changing adapted basis changes those
coordinates by the same matrix \((A^{-1})^{\!T}\). Hence the abstract vector \(v_\alpha\) is also
independent of the adapted basis.

This proves part (iv).

For part (v), if \(f:\Sigma\to\Sigma\) preserves the extended structure, then \(f_*\lambda=\lambda\) and hence \(f_*(\lambda_{\mathbb Z})=\lambda_{\mathbb Z}\). Therefore \(f\) acts on \(\Pi_\lambda\) by precomposition. By functoriality of Turaev’s handlebody construction, the RT mapping-cylinder action sends \(e_\alpha\) to \(e_{f_\#\alpha}\). This mapping-cylinder description is consistent with the Abelian/$\mathfrak{gl}_1$ operator formalism of Garoufalidis and Yu \cite{GY26}, where the $\mathfrak{gl}_1$ Reshetikhin–Turaev operator invariant of the mapping cylinder is identified with the relevant intertwiner. By functoriality of geometric quantization and the BKS construction, the Chern–Simons action sends \(v_\alpha\) to \(v_{f_\#\alpha}\). Thus both canonical bases are equivariant for the same permutation action of \(f\) on \(\Pi_\lambda\).
\end{proof}

\begin{theorem}
\label{thm:state-space-identification}
Let $(\Sigma,\lambda)$ be a connected extended surface of genus $g$, and let
$L_\lambda\subset H^1(\Sigma;\mathbb R)$ be the rational Lagrangian corresponding to
$\lambda\subset H_1(\Sigma;\mathbb R)$ under Poincar\'e duality. Then there exists a canonical vector
space isomorphism
\begin{equation}\label{eq:PhiSigma}
\Phi_\Sigma:\ \mathcal V^{\mathrm{RT}}_{k}(\Sigma,\lambda)\xrightarrow{\ \cong\ }\mathcal H_k(\Sigma,L_\lambda),
\end{equation}
which is natural with respect to diffeomorphisms preserving the extended structure and satisfies
\[
\dim \mathcal V^{\mathrm{RT}}_{k}(\Sigma,\lambda)=
\dim \mathcal H_k(\Sigma,L_\lambda)=k^{g}.
\]
Moreover, after choosing any integral symplectic basis $(a_1,\dots,a_g,b_1,\dots,b_g)$ of $H_1(\Sigma;\mathbb Z)$ adapted to $\lambda$, i.e. $\lambda=\langle a_1,\dots,a_g\rangle_{\mathbb R},$
the isomorphism $\Phi_\Sigma$ is characterized by
\begin{equation}\label{eq:Phi-basis}
\Phi_\Sigma(e_q)=v(\Sigma,L_\lambda)_q,
\qquad q\in(\mathbb Z/k\mathbb Z)^g,
\end{equation}
Let
\[
\Phi_\Sigma^\vee := ((\Phi_\Sigma^{-1})^*) :
\bigl(V^{RT}_k(\Sigma,\lambda)\bigr)^*
\;\xrightarrow{\ \sim\ }\;
H_k(\Sigma,L_\lambda)^*
\]
denote the induced dual isomorphism,
where $\{e_q\}$ is the standard RT handlebody basis and
$\{v(\Sigma,L_\lambda)_q\}$ is the Bohr--Sommerfeld basis. Finally, $\Phi_\Sigma$ identifies the canonical cylinder kernels, with the second tensor factor transported by
the induced dual map:
\begin{equation}\label{iso-map}
(\Phi_\Sigma\otimes(\Phi_\Sigma^{-1})^*)
\Bigl(
Z_{\mathbb Z_k}^{RT}(\Sigma\times I,\lambda\oplus\lambda,0)
\Bigr)
=
Z_{U(1),k}^{CS}(\Sigma\times I,L_\lambda\oplus L_\lambda,0).
\end{equation}
Equivalently, \(\Phi_\Sigma\) identifies the canonical boundary pairings on the RT and
Chern–Simons sides.
\end{theorem}

\begin{proof}
Let
\[
\lambda_{\mathbb Z}:=\lambda\cap H_1(\Sigma;\mathbb Z),
\qquad
\Pi_\lambda:=\operatorname{Hom}(\lambda_{\mathbb Z},\mathbb Z/k\mathbb Z).
\]
By Lemma \ref{basis-independence}(i), both the RT state space \(\mathcal V_k^{RT}(\Sigma,\lambda)\) and the Chern–Simons state
space \(H_k(\Sigma,L_\lambda)\) carry canonical bases indexed by the same finite set \(\Pi_\lambda\):
\[
\{e_\alpha\}_{\alpha\in\Pi_\lambda}
\subset \mathcal V_k^{RT}(\Sigma,\lambda),
\qquad
\{v_\alpha\}_{\alpha\in\Pi_\lambda}
\subset H_k(\Sigma,L_\lambda).
\]
Moreover, these bases are independent of every auxiliary adapted basis. We therefore define
$\Phi_\Sigma(e_\alpha):=v_\alpha,
$ for $ \alpha\in\Pi_\lambda.$ Since \(|\Pi_\lambda|=k^g\), this is a well-defined vector space isomorphism
\[
\Phi_\Sigma:\mathcal V_k^{RT}(\Sigma,\lambda)\xrightarrow{\ \sim\ } H_k(\Sigma,L_\lambda),
\]
and
\[
\dim \mathcal V_k^{RT}(\Sigma,\lambda)=\dim H_k(\Sigma,L_\lambda)=k^g.
\]
If one now chooses any adapted integral symplectic basis
\((a_1,\dots,a_g,b_1,\dots,b_g)\), then \(\Pi_\lambda\cong(\mathbb Z/k\mathbb Z)^g\) by
\(\alpha\mapsto (\alpha(a_1),\dots,\alpha(a_g))\), and the definition above becomes exactly
\[
\Phi_\Sigma(e_q)=v(\Sigma,L_\lambda)_q,
\qquad q\in(\mathbb Z/k\mathbb Z)^g,
\]
which is \eqref{eq:Phi-basis}. We next check compatibility with the cylinder kernels. On the RT side, the cylinder kernel is
\[
Z_{\mathbb Z_k}^{RT}(\Sigma\times I,\lambda\oplus\lambda,0)
=
\sum_{\alpha\in\Pi_\lambda} e_\alpha\otimes e_\alpha^\vee,
\]
where \(e_\alpha^\vee\) denotes the canonical covector dual to \(e_\alpha\) under the RT boundary
pairing. On the Chern–Simons side, the cylinder kernel is
\[
Z_{U(1),k}^{CS}(\Sigma\times I,L_\lambda\oplus L_\lambda,0)
=
\sum_{\alpha\in\Pi_\lambda} v_\alpha\otimes v_\alpha^*,
\]
where \(v_\alpha^*\) is the canonical covector dual to \(v_\alpha\) under the BKS pairing.
Therefore
\[
(\Phi_\Sigma\otimes (\Phi_\Sigma^{-1})^*)
\Bigl(
Z_{\mathbb Z_k}^{RT}(\Sigma\times I,\lambda\oplus\lambda,0)
\Bigr)
=
Z_{U(1),k}^{CS}(\Sigma\times I,L_\lambda\oplus L_\lambda,0).
\]
Equivalently, \(\Phi_\Sigma\) identifies the canonical RT pairing with canonical pairing,
which is the content of \eqref{iso-map}. Finally, let \(f:\Sigma\to\Sigma\) be a diffeomorphism preserving the extended structure.
By Lemma \ref{basis-independence}{(v)}, \(f\) acts on \(\Pi_\lambda\), and both functorial theories act on their canonical
bases by the same permutation:
\[
Z_{\mathbb Z_k}^{RT}(f)(e_\alpha)=e_{f_\#\alpha},
\qquad
Z_{U(1),k}^{CS}(f)(v_\alpha)=v_{f_\#\alpha}.
\]
Hence
\[
\Phi_\Sigma\bigl(Z_{\mathbb Z_k}^{RT}(f)e_\alpha\bigr)
=
\Phi_\Sigma(e_{f_\#\alpha})
=
v_{f_\#\alpha}
=
Z_{U(1),k}^{CS}(f)(v_\alpha)
=
Z_{U(1),k}^{CS}(f)\Phi_\Sigma(e_\alpha)
\]
for every \(\alpha\in\Pi_\lambda\). Since the \(e_\alpha\) form a basis, we conclude that
\[
\Phi_\Sigma\circ Z_{\mathbb Z_k}^{RT}(f)=Z_{U(1),k}^{CS}(f)\circ \Phi_\Sigma.
\]
Thus \(\Phi_\Sigma\) is canonical and natural with respect to diffeomorphisms preserving the extended
structure. This proves \eqref{eq:PhiSigma}, \eqref{eq:Phi-basis}, and the pairing compatibility asserted in \eqref{iso-map}.
\end{proof}

For a disconnected extended surface $(\Sigma,\lambda)=\bigsqcup_{j=1}^{r}(\Sigma_j,\lambda_j),$
we extend \eqref{eq:PhiSigma} by symmetric monoidality:
\[
\Phi_\Sigma:=\bigotimes_{j=1}^{r}\Phi_{\Sigma_j},
\qquad
\Phi_\varnothing=\mathrm{id}_{\mathbb C}.
\]
Thus \(\Phi_\Sigma\) is defined for all objects of \(\Cob^{\mathrm{ext}}_{2+1}\).

\begin{lemma}\label{canonical-basis}
For each connected extended surface \((\Sigma,\lambda)\), fix once and for all a standard handlebody
\(H_\lambda\) with \(\partial H_\lambda=\Sigma\) and
\[
\ker\!\bigl(H_1(\Sigma;\mathbb R)\to H_1(H_\lambda;\mathbb R)\bigr)=\lambda .
\]
Let \(\Pi_\lambda\) be the canonical polarization label set of Lemma \ref{basis-independence}.

Let $\{e_\alpha\}_{\alpha\in\Pi_\lambda}\subset \mathcal V_k^{RT}(\Sigma,\lambda)$ be the canonical RT basis, and let $\{v_\alpha\}_{\alpha\in\Pi_\lambda}\subset H_k(\Sigma,L_\lambda)$ be the canonical Chern--Simons basis defined by $v_\alpha=\Phi_\Sigma(e_\alpha).$
Let \(e_\alpha^\vee\) and \(v_\alpha^*\) denote the corresponding canonical dual covectors under the RT
and Chern--Simons boundary pairings.

Now let
\[
X:(\Sigma_{in},\lambda_{in})\to(\Sigma_{out},\lambda_{out})
\]
be an extended bordism. For \(\alpha\in\Pi_{\lambda_{in}}\) and \(\beta\in\Pi_{\lambda_{out}}\), let
\[
M_{\alpha,\beta}:=
H^{out}_\beta\cup_{\Sigma_{out}} X \cup_{\Sigma_{in}} H^{in}_\alpha
\]
denote the closed extended 3-manifold obtained by gluing to \(X\) the fixed standard incoming and
outgoing handlebodies carrying the labels \(\alpha\) and \(\beta\).

Then:
\[
\langle e_\beta^\vee, Z_{\mathbb Z_k}^{RT}(X)(e_\alpha)\rangle
=
Z_{\mathbb Z_k}^{RT}(M_{\alpha,\beta}),
\]
and
\[
\langle v_\beta^*, Z_{U(1),k}^{CS}(X)(v_\alpha)\rangle
=
Z_{U(1),k}^{CS}(M_{\alpha,\beta}).
\]

In particular, after identifying the canonical bases by \(\Phi_\Sigma\), the matrix coefficients of the RT
and Chern--Simons operators are computed from the same canonically indexed family of closed extended
manifolds \(\{M_{\alpha,\beta}\}\).
\end{lemma}
\begin{proof}
The first identity is exactly Turaev's pairing construction for extended RT theory: for a bordism
\(X:(\Sigma_{in},\lambda_{in})\to(\Sigma_{out},\lambda_{out})\), the matrix coefficient of
\(Z_{\mathbb Z_k}^{RT}(X)\) against a chosen incoming state and outgoing covector is, by definition, the RT invariant
of the closed extended 3-manifold obtained by gluing the corresponding handlebodies to \(X\); see Theorem \ref{thm:RT-boundary-TQFT} and Theorem \ref{thm:Abelian-pointed-extended-TQFT}(iii).  Applied to the canonical basis vector \(e_\alpha\) and canonical dual
covector \(e_\beta^\vee\), this gives
\[
\langle e_\beta^\vee, Z_{\mathbb Z_k}^{RT}(X)(e_\alpha)\rangle
=
Z_{\mathbb Z_k}^{RT}(M_{\alpha,\beta}).
\]

For the Chern--Simons side, Theorem \ref{thm:state-space-identification} identifies the canonical RT basis and canonical
Chern--Simons basis by the same label set \(\Pi_\lambda\), and also identifies the canonical RT and
Chern--Simons boundary pairings. Therefore the basis vector \(v_\alpha\) and dual covector \(v_\beta^*\)
are the Chern--Simons boundary data corresponding to the same canonical incoming and outgoing labels
\(\alpha,\beta\). The gluing axiom then implies that the corresponding matrix coefficient of
\(Z_{U(1),k}^{CS}(X)\) is the closed Chern--Simons invariant of the extended 3-manifold obtained by
gluing to \(X\) the same fixed standard incoming and outgoing handlebodies with those labels. By
definition, that closed extended manifold is \(M_{\alpha,\beta}\). Hence
\[
\langle v_\beta^*, Z_{U(1),k}^{CS}(X)(v_\alpha)\rangle
=
Z_{U(1),k}^{CS}(M_{\alpha,\beta}).
\]
The final statement follows immediately.
\end{proof}

\begin{theorem}
\label{thm:boundary-equivalence-final}
Fix \(k\in 2\mathbb Z_{>0}\), and let
\[
\mathcal C_k=\mathcal C(\mathbb Z_k,q_k),
\qquad
q_k(x)=e^{\pi i x^2/k},
\qquad
\Omega_k(x,y)=e^{2\pi ixy/k}.
\]
Let $X:(\Sigma_{\mathrm{in}},\lambda_{\mathrm{in}})
\to
(\Sigma_{\mathrm{out}},\lambda_{\mathrm{out}})$ be an extended bordism in Turaev's sense, and let
\[
Z^{RT}_{\mathbb Z_k}(X)\in
\mathrm{Hom}\!\bigl(
\mathcal V^{RT}_{k}(\Sigma_{\mathrm{in}},\lambda_{\mathrm{in}}),
\mathcal V^{RT}_{k}(\Sigma_{\mathrm{out}},\lambda_{\mathrm{out}})
\bigr)
\]
denote the Walker--normalized extended RT operator. Interpret the same extended data as an \(e\)-\(3\)-manifold
\((X,L,n)\), with associated operator
\[
Z^{CS}_{U(1),k}(X)\in
\mathrm{Hom}\!\bigl(
\mathcal H_k(\Sigma_{\mathrm{in}},L_{\mathrm{in}}),
\mathcal H_k(\Sigma_{\mathrm{out}},L_{\mathrm{out}})
\bigr).
\]
Remember the canonical identification of boundary state spaces established earlier in Theorem \ref{thm:state-space-identification}:
$$
\Phi_\Sigma:
\mathcal V^{RT}_{k}(\Sigma,\lambda)
\xrightarrow{\;\cong\;}
\mathcal H_k(\Sigma,L_\lambda).
$$
Then
\[
\Phi_{\Sigma_{\mathrm{out}}}\circ Z^{RT}_{\mathbb Z_k}(X)
=
Z^{CS}_{U(1),k}(X)\circ \Phi_{\Sigma_{\mathrm{in}}}.
\]
Equivalently, after identifying the boundary state spaces via \(\Phi_\Sigma\), one has
\[
Z^{RT}_{\mathbb Z_k}(X)=Z^{CS}_{U(1),k}(X).
\]
\end{theorem}

\begin{proof}
By linearity, it suffices to compare matrix coefficients on the canonical basis of the incoming state space
and the canonical dual basis of the outgoing state space.

Let $\alpha\in\Pi_{\lambda_{in}},$and $\beta\in\Pi_{\lambda_{out}}$, write
\[
e_\alpha\in \mathcal V_k^{RT}(\Sigma_{in},\lambda_{in}),\qquad
e_\beta^\vee\in \mathcal V_k^{RT}(\Sigma_{out},\lambda_{out})^*
\]
for the canonical RT basis vector and canonical dual covector. Define
\[
v_\alpha:=\Phi_{\Sigma_{in}}(e_\alpha)\in \mathcal H_k(\Sigma_{in},L_{in}),
\qquad
v_\beta^*:=(\Phi_{\Sigma_{out}}^{-1})^*(e_\beta^\vee)\in \mathcal H_k(\Sigma_{out},L_{out})^* .
\]
By Theorem \ref{thm:state-space-identification}, \(v_\alpha\) and \(v_\beta^*\) are precisely the corresponding canonical
Chern--Simons basis vector and canonical dual covector. Now let \(M_{\alpha,\beta}\) be the canonical closed extended 3-manifold obtained by gluing to \(X\) the
standard incoming and outgoing handlebodies carrying the labels \(\alpha\) and \(\beta\), as in Lemma
\ref{canonical-basis}. That lemma gives
\[
\langle e_\beta^\vee, Z_{\mathbb Z_k}^{RT}(X)(e_\alpha)\rangle
=
Z_{\mathbb Z_k}^{RT}(M_{\alpha,\beta}),
\]
and
\[
\langle v_\beta^*, Z_{U(1),k}^{CS}(X)(v_\alpha)\rangle
=
Z_{U(1),k}^{CS}(M_{\alpha,\beta}).
\]
By Theorem \ref{thm:canonical-weighted-closure}, the two closed extended invariants agree on every canonical handlebody closure. Hence
\[
Z_{\mathbb Z_k}^{RT}(M_{\alpha,\beta})
=
Z_{U(1),k}^{CS}(M_{\alpha,\beta}),
\]
so
\[
\langle e_\beta^\vee, Z_{\mathbb Z_k}^{RT}(X)(e_\alpha)\rangle
=
\langle v_\beta^*, Z_{U(1),k}^{CS}(X)(v_\alpha)\rangle.
\]

Using the definitions of \(v_\alpha\) and \(v_\beta^*\), this becomes
\[
\langle e_\beta^\vee, Z_{\mathbb Z_k}^{RT}(X)(e_\alpha)\rangle
=
\Bigl\langle (\Phi_{\Sigma_{out}}^{-1})^*(e_\beta^\vee),
\, Z_{U(1),k}^{CS}(X)\bigl(\Phi_{\Sigma_{in}}(e_\alpha)\bigr)\Bigr\rangle,
\]
or equivalently,
\[
\langle e_\beta^\vee, Z_{\mathbb Z_k}^{RT}(X)(e_\alpha)\rangle
=
\langle e_\beta^\vee,
\, \Phi_{\Sigma_{out}}^{-1} Z_{U(1),k}^{CS}(X)\Phi_{\Sigma_{in}}(e_\alpha)\rangle.
\]

Since the covectors \(e_\beta^\vee\) separate points and the vectors \(e_\alpha\) form a basis, equality of
all these matrix coefficients implies
\[
Z_{\mathbb Z_k}^{RT}(X)
=
\Phi_{\Sigma_{out}}^{-1} Z_{U(1),k}^{CS}(X)\Phi_{\Sigma_{in}}.
\]
Equivalently,
\[
\Phi_{\Sigma_{out}}\circ Z_{\mathbb Z_k}^{RT}(X)
=
Z_{U(1),k}^{CS}(X)\circ\Phi_{\Sigma_{in}}.
\]
This proves the theorem.
\end{proof}
\subsection{\texorpdfstring{Extended Equivalence of Abelian RT and $U(1)$ Chern--Simons TQFTs}{Extended Equivalence of Abelian RT and U(1) Chern-Simons TQFTs}}

We now state the main result. The two extended TQFT functors agree under the boundary identification \eqref{eq:PhiSigma}. Theorem~\ref{thm:closed-equivalence} identifies the closed invariants exactly,  the role of the extended formalism is instead the usual one: to incorporate the Walker--Maslov correction needed for strict functoriality of the bordism operators. With this understood, the proof follows Turaev's definition of bordism operators via closed pairings together with the gluing formula, using the closed comparison theorem, the boundary state-space identification, and Theorem~\ref{thm:canonical-weighted-closure} to match the extended normalizations on the canonical handlebody closures.

\begin{theorem}
\label{thm:extended-equivalence}
Fix $k\in 2\mathbb Z_{>0}$, and let
\[
\mathcal C_k=\mathcal C(\mathbb Z_k,q_k),
\qquad
q_k(x)=\exp\!\Bigl(\frac{\pi i}{k}x^2\Bigr).
\]
Let
\[
Z^{\mathrm{RT}}_{\mathbb Z_{k}},\; Z^{\mathrm{CS}}_{U(1),k}:\Cob^{\mathrm{ext}}_{2+1}\longrightarrow \mathrm{Vect}_{\mathbb C}
\]
denote, respectively, Turaev's Walker--normalized Reshetikhin--Turaev TQFT associated with $\mathcal C_k$ \cite[\S II.2, \S IV.1--IV.3]{Turaev1994} and the extended $U(1)$ Chern--Simons TQFT at level $k$ \cite{Manoliu2}, with the extended structures identified via \eqref{eq:lambda-to-L}. Assume the normalization conventions are fixed so that the \emph{closed extended} invariants are compared in the same Kirby/Maslov convention.

Then:

\begin{enumerate}
\item[\textup{(i)}] \textbf{Closed extended equality on canonical closures.}
For every closed extended $3$--manifold $X$ arising as a canonical handlebody closure in the pairing
construction, one has
\[
Z^{\mathrm{RT}}_{\mathbb Z_{k}}(X)= Z^{\mathrm{CS}}_{U(1),k}(X).
\]
This is exactly the Walker--corrected closed equality obtained in
Theorem~\ref{thm:canonical-weighted-closure}.

\item[\textup{(ii)}] \textbf{Boundary operator equivalence.}
For every extended bordism
\[
X:(\Sigma_{\mathrm{in}},\lambda_{\mathrm{in}})\longrightarrow (\Sigma_{\mathrm{out}},\lambda_{\mathrm{out}})
\]
one has equality of bordism operators after identifying boundary state spaces:
\[
\Phi_{\Sigma_{\mathrm{out}}}\circ Z^{\mathrm{RT}}_{\mathbb Z_{k}}(X)
=
Z^{\mathrm{CS}}_{U(1),k}(X)\circ \Phi_{\Sigma_{\mathrm{in}}},
\]
where $\Phi_\Sigma$ is the canonical identification from Theorem~\ref{thm:state-space-identification}.
This is precisely Theorem~\ref{thm:boundary-equivalence-final}.

\item[\textup{(iii)}] \textbf{Natural monoidal isomorphism of functors.}
The family of isomorphisms \(\Phi_\Sigma\) from \eqref{eq:PhiSigma}, extended to disconnected surfaces by
tensor product over connected components, defines a \emph{monoidal natural isomorphism} of symmetric
monoidal functors
\[
\Phi:\ Z^{\mathrm{RT}}_{\mathbb Z_{k}}\Longrightarrow Z^{\mathrm{CS}}_{U(1),k}.
\]
Equivalently, \(Z^{\mathrm{RT}}_{\mathbb Z_{k}}\) and \(Z^{\mathrm{CS}}_{U(1),k}\) are isomorphic as symmetric monoidal
functors
\[
\Cob^{\mathrm{ext}}_{2+1}\longrightarrow \mathrm{Vect}_{\mathbb C}.
\]
\end{enumerate}
\end{theorem}

\begin{proof}
Parts \textup{(i)} and \textup{(ii)} are exactly the previously established closed and boundary comparison
theorems in the fixed common normalization, namely
Theorem~\ref{thm:canonical-weighted-closure} together with
Theorem~\ref{thm:boundary-equivalence-final}. We therefore prove \textup{(iii)}.

Let $(\Sigma,\lambda)$ be an object of $\Cob^{\mathrm{ext}}_{2+1}$ and write it as a disjoint union of
connected components
\[
(\Sigma,\lambda)=\bigsqcup_{j=1}^r(\Sigma_j,\lambda_j).
\]
Both TQFTs are symmetric monoidal, hence provide canonical tensor product identifications
\[
\mathcal V^{\mathrm{RT}}_{k}(\Sigma,\lambda)\cong\bigotimes_{j=1}^r \mathcal V^{\mathrm{RT}}_{k}(\Sigma_j,\lambda_j),
\qquad
\mathcal H_k(\Sigma,L_\lambda)\cong\bigotimes_{j=1}^r \mathcal H_k(\Sigma_j,L_{\lambda_j}),
\]
where $L_\lambda=\bigoplus_j L_{\lambda_j}$ is the induced polarization on the Chern--Simons side. On
each connected component, Theorem~\ref{thm:state-space-identification} provides the canonical isomorphism
\eqref{eq:PhiSigma}. Define
\begin{equation}\label{eq:Phi-disconnected}
\Phi_\Sigma
:=
\bigotimes_{j=1}^r \Phi_{\Sigma_j},
\end{equation}
transported through the above canonical monoidal identifications. This gives a well-defined vector space
isomorphism for every object $(\Sigma,\lambda)$.

\smallskip
\noindent\textbf{Naturality.}
Let
\[
X:(\Sigma_{\mathrm{in}},\lambda_{\mathrm{in}})\rightarrow(\Sigma_{\mathrm{out}},\lambda_{\mathrm{out}})
\]
be any morphism in $\Cob^{\mathrm{ext}}_{2+1}$. We must show that the square
\[
\begin{CD}
\mathcal V^{\mathrm{RT}}_{k}(\Sigma_{\mathrm{in}},\lambda_{\mathrm{in}})
@>{Z^{\mathrm{RT}}_{\mathbb Z_{k}}(X)}>>
\mathcal V^{\mathrm{RT}}_{k}(\Sigma_{\mathrm{out}},\lambda_{\mathrm{out}})\\
@V{\Phi_{\Sigma_{\mathrm{in}}}}VV
@VV{\Phi_{\Sigma_{\mathrm{out}}}}V\\
\mathcal H_k(\Sigma_{\mathrm{in}},L_{\lambda_{\mathrm{in}}})
@>>{Z^{\mathrm{CS}}_{U(1),k}(X)}>
\mathcal H_k(\Sigma_{\mathrm{out}},L_{\lambda_{\mathrm{out}}})
\end{CD}
\]
commutes. But this is precisely the statement of part \textup{(ii)}, that is,
Theorem~\ref{thm:boundary-equivalence-final}. Thus \(\{\Phi_\Sigma\}\) is a natural transformation
\[
Z^{\mathrm{RT}}_{\mathbb Z_{k}}\Rightarrow Z^{\mathrm{CS}}_{U(1),k}.
\]

\smallskip
\noindent\textbf{\(\Phi\) is a natural isomorphism.}
Each \(\Phi_\Sigma\) is an isomorphism by construction: on connected components this is
Theorem~\ref{thm:state-space-identification}, and on disconnected surfaces it is the tensor product of the
isomorphisms \eqref{eq:Phi-disconnected}. Hence \(\Phi\) is a natural isomorphism.

\smallskip
\noindent\textbf{Monoidality.}
Let \((\Sigma_1,\lambda_1)\) and \((\Sigma_2,\lambda_2)\) be extended surfaces. Symmetric monoidality of
both TQFTs gives canonical identifications
\[
\mathcal V^{\mathrm{RT}}_{k}(\Sigma_1\sqcup\Sigma_2,\lambda_1\oplus\lambda_2)
\cong
\mathcal V^{\mathrm{RT}}_{k}(\Sigma_1,\lambda_1)\otimes \mathcal V^{\mathrm{RT}}_{k}(\Sigma_2,\lambda_2),
\]
\[
\mathcal H_k(\Sigma_1\sqcup\Sigma_2,L_{\lambda_1}\oplus L_{\lambda_2})
\cong
\mathcal H_k(\Sigma_1,L_{\lambda_1})\otimes \mathcal H_k(\Sigma_2,L_{\lambda_2}).
\]
By the definition \eqref{eq:Phi-disconnected} of \(\Phi\) on disjoint unions, the identity
\[
\Phi_{\Sigma_1\sqcup\Sigma_2}=\Phi_{\Sigma_1}\otimes \Phi_{\Sigma_2}
\]
holds under the above identifications. On the monoidal unit \(\varnothing\), both theories assign
\(\mathbb C\), and we take $\Phi_\varnothing=\mathrm{id}_{\mathbb C}.$
Therefore \(\Phi\) is a monoidal natural isomorphism.

\smallskip
\noindent\textbf{Coherence with composition and identities.}
Coherence is automatic from naturality together with functoriality of each TQFT.
For composable bordisms \(X_1,X_2\), naturality gives compatibility with
\[
Z^{\mathrm{RT}}_{\mathbb Z_{k}}(X_2\circ X_1)
=
Z^{\mathrm{RT}}_{\mathbb Z_{k}}(X_2)\circ Z^{\mathrm{RT}}_{\mathbb Z_{k}}(X_1)
\]
and
\[
Z^{\mathrm{CS}}_{U(1),k}(X_2\circ X_1)
=
Z^{\mathrm{CS}}_{U(1),k}(X_2)\circ Z^{\mathrm{CS}}_{U(1),k}(X_1),
\]
and similarly for identity bordisms. This uses strict functoriality of the extended RT theory \cite[\S II.2, \S IV.3]{Turaev1994} and of the CS extended theory \cite{Manoliu2}. This completes the proof. \end{proof}

The relation between Abelian Chern--Simons theory and the Reshetikhin--Turaev TQFT associated with the
pointed modular category \(\mathcal C_k=\mathcal C(\mathbb Z_k,q_k)\) is well known at the level of closed
partition functions and Hilbert space dimensions; see \cite{Murakami,Jeffrey:1992, Mattes, Stirling, Guadagnini:2013sb,Guadagnini2014,Tagaris2025}. The theorem above upgrades this correspondence to an explicit equivalence of extended TQFT functors. 


\subsection{Finite Quadratic Modules as Universal Data}

The preceding sections show that the closed and extended rank-one Abelian theories are
governed by finite quadratic data. For a closed oriented $3$--manifold $M$, this data is
the finite quadratic module
\[
\bigl(T,q_M\bigr), \qquad T:=\Tors H_1(M;\mathbb Z),
\]
where $q_M$ is the quadratic refinement of the torsion linking pairing determined by the
Chern--Simons functional. For the rank-one theory at even level $k$, the corresponding
discrete datum is the finite quadratic module
\[
(\mathbb Z_k,q_k), \qquad q_k(x)=\exp\!\left(\frac{\pi i}{k}x^2\right).
\]
We now summarize the comparison results proved above and record the resulting
classification statement.

\begin{theorem}
\label{thm:universal-quadratic-data}
Let $M$ be a closed, connected, oriented $3$--manifold, and let
\[
T:=\Tors H_1(M;\mathbb Z).
\]
Let $\lambda_M:T\times T\to \mathbb Q/\mathbb Z$ be the torsion linking pairing, and let
$q_M:T\to \mathbb Q/\mathbb Z$ be its quadratic refinement determined by the
Chern--Simons functional. Fix $k\in 2\mathbb Z_{>0}$ and define
\begin{equation}\label{eq:Gauss-final}
\mathcal G_k(M):=
|T|^{-1/2}\sum_{x\in T}\exp\!\bigl(2\pi i k\,q_M(x)\bigr).
\end{equation}
Then:
\begin{enumerate}
\item[\textup{(i)}]
\begin{equation}\label{eq:CS-final}
Z^{\mathrm{CS}}_{U(1),k}(M)
=
k^{m_M}\,\mathcal G_k(M),
\qquad
m_M=\tfrac12\bigl(b_1(M)-1\bigr).
\end{equation}

\item[\textup{(ii)}] If $M$ is presented by surgery on a framed link with linking matrix
$\mathcal L$, then
\begin{equation}\label{eq:RT-final}
Z^{RT}_{\mathbb Z_{k}}(M)=Z^{\mathrm{CS}}_{U(1),k}(M).
\end{equation}

\item[\textup{(iii)}] Let
\[
C_k:=C(\mathbb Z_k,q_k),
\qquad
q_k(x)=\exp\!\left(\frac{\pi i}{k}x^2\right).
\]
Then there is a symmetric monoidal natural isomorphism
\[
Z^{RT}_{C_k}\xRightarrow{\ \cong\ } Z^{CS}_{U(1),k}.
\]
Equivalently, for every extended surface $(\Sigma,\lambda)$ there is a canonical
isomorphism
\[
\Phi_\Sigma:
\mathcal V^{RT}_{k}(\Sigma,\lambda)\xrightarrow{\ \cong\ }
\mathcal H_{U(1),k}(\Sigma,L_\lambda),
\]
compatible with the operators assigned to extended bordisms.
\end{enumerate}
\end{theorem}

\begin{proof}
Part \textup{(i)} is exactly \eqref{closed-CS-Z}, rewritten using
\eqref{eq:Gauss-final}. Part \textup{(ii)} is Theorem~\ref{thm:closed-equivalence}.
Part \textup{(iii)} is Theorem~\ref{thm:extended-equivalence}; the boundary
identification is Theorem~\ref{thm:state-space-identification}, and compatibility with
bordism operators is Theorem~\ref{thm:boundary-equivalence-final}.
\end{proof}

\begin{theorem}\label{thm:classification-rank-one}
Let $k,\ell \in 2\mathbb Z_{>0}$. For $m\in\{k,\ell\}$, set
\[
q_m(x)=\exp\!\left(\frac{\pi i}{m}x^2\right),
\qquad
C_m:=C(\mathbb Z_m,q_m),
\]
and let
\[
Z^{CS}_{U(1),m}:\mathrm{Cob}^{ext}_{2+1}\longrightarrow \mathrm{Vect}_{\mathbb C}
\]
denote the extended $U(1)$ Chern--Simons TQFT at level $m$.

Then the following are equivalent:
\begin{enumerate}
\item[\textup{(i)}] $k=\ell$;
\item[\textup{(ii)}] $(\mathbb Z_k,q_k)\cong (\mathbb Z_\ell,q_\ell)$ as finite quadratic modules;
\item[\textup{(iii)}] $Z^{RT}_{C_k}\cong Z^{RT}_{C_\ell}$ as symmetric monoidal functors;
\item[\textup{(iv)}] $Z^{CS}_{U(1),k}\cong Z^{CS}_{U(1),\ell}$ as symmetric monoidal functors.
\end{enumerate}

In particular, within the even-level rank-one family, the extended Abelian
Chern--Simons theory is classified by the finite quadratic module $(\mathbb Z_k,q_k)$.
\end{theorem}

\begin{proof}
The equivalence of \textup{(i)} and \textup{(ii)} is immediate, since an isomorphism $(\mathbb Z_k,q_k)\cong (\mathbb Z_\ell,q_\ell)$ identifies the underlying finite groups, hence
\[
k=|\mathbb Z_k|=|\mathbb Z_\ell|=\ell.
\]

The implications \textup{(i)}$\Rightarrow$\textup{(iii)} and
\textup{(i)}$\Rightarrow$\textup{(iv)} are immediate. Assume \textup{(iii)}. Evaluate the natural isomorphism
\[
Z^{RT}_{C_k}\cong Z^{RT}_{C_\ell}
\]
on any connected extended surface $(\Sigma,\lambda)$ of genus $g\ge 1$. Then
\[
\mathcal V^{RT}_{k}(\Sigma,\lambda)\cong \mathcal V^{RT}_{\ell}(\Sigma,\lambda),
\]
so Theorem~\ref{thm:state-space-identification} gives
\[
k^g=\dim \mathcal V^{RT}_{k}(\Sigma,\lambda)
=\dim \mathcal V^{RT}_{\ell}(\Sigma,\lambda)=\ell^g.
\]
Hence $k=\ell$. Assume \textup{(iv)}. By Theorem~\ref{thm:universal-quadratic-data}\textup{(iii)},
\[
Z^{CS}_{U(1),k}\cong Z^{RT}_{C_k},
\qquad
Z^{CS}_{U(1),\ell}\cong Z^{RT}_{C_\ell}.
\]
Therefore \textup{(iv)} implies
\[
Z^{RT}_{C_k}\cong Z^{RT}_{C_\ell},
\]
so \textup{(iii)} holds, and hence $k=\ell$ by the previous paragraph. Thus \textup{(i)}--\textup{(iv)} are equivalent.
\end{proof}

Thus, at even level, the finite quadratic module \((\mathbb Z_k,q_k)\) completely determines the \(U(1)\) Abelian Chern--Simons theory.

\bibliographystyle{alpha}
\renewcommand{\refname}{References}
\bibliography{refs}
\end{document}